\documentclass[reqno, 10pt]{amsart}
\usepackage{amssymb}
\usepackage{amsmath}
\usepackage[mathscr]{euscript}
\usepackage{hyperref}
\usepackage{graphicx}
\usepackage[normalem]{ulem}

\usepackage{amssymb}
\usepackage{hyperref}
\RequirePackage[dvipsnames]{xcolor} 
\definecolor{halfgray}{gray}{0.55} 
\definecolor{webgreen}{rgb}{0,0.5,0}
\definecolor{webbrown}{rgb}{.6,0,0} \hypersetup{%
  colorlinks=true, linktocpage=true, pdfstartpage=3,
  pdfstartview=FitV,%
  breaklinks=true, pdfpagemode=UseNone, pageanchor=true,
  pdfpagemode=UseOutlines,%
  plainpages=false, bookmarksnumbered, bookmarksopen=true,
  bookmarksopenlevel=1,%
  hypertexnames=true,
  pdfhighlight=/O,
  urlcolor=webbrown, linkcolor=RoyalBlue,
  citecolor=webgreen, 
  pdftitle={},%
  pdfauthor={},%
  pdfsubject={2000 MAthematical Subject Classification: Primary:},%
  pdfkeywords={},%
  pdfcreator={pdfLaTeX},%
  pdfproducer={LaTeX with hyperref}%
}

\usepackage{enumitem}

\theoremstyle{plain}
\newtheorem{theorem}{Theorem}[section]
\newtheorem{lemma}[theorem]{Lemma}
\newtheorem{corollary}[theorem]{Corollary}
\newtheorem{proposition}[theorem]{Proposition}

\theoremstyle{definition}

\newtheorem{remark}[theorem]{Remark}
\newtheorem{example}[theorem]{Example}
\DeclareMathOperator{\Ima}{Im}
\DeclareMathOperator{\Ker}{Ker}
\DeclareMathOperator{\Id}{Id}
\DeclareMathOperator*{\essinf}{ess\,inf}

\def\R{\mathbb{R}}
\def\N{\mathbb{N}}
\def\Reg{\mathcal{R}^{\mu}}
\def\B{\mathcal{B}}
\newcommand{\norm}[1]{{\left\lVert \, #1 \, \right\rVert}}

\begin{document}

\title
[Periodic approximation of exceptional Lyapunov exponents]
{Periodic approximation of exceptional Lyapunov exponents for semi-invertible operator cocycles}

\author{Lucas Backes}

\address{Departamento de Matem\'atica, Universidade Federal do Rio Grande do Sul, Av. Bento Gon\c{c}alves 9500, CEP 91509-900, Porto Alegre, RS, Brazil.}
\email{lhbackes@impa.br}

\author{Davor Dragi\v cevi\'c}
\address{Department of Mathematics, University of Rijeka, 51000 Rijeka, Croatia}
\email{ddragicevic@math.uniri.hr}

\date{\today}


\begin{abstract}

We prove that for semi-invertible and H\"{o}lder continuous linear cocycles $A$ acting on an arbitrary Banach space and defined over a base space that satisfies the Anosov Closing Property, all exceptional Lyapunov exponents of $A$ with respect to 
an ergodic invariant measure for base dynamics can be approximated with Lyapunov exponents of $A$ with respect to ergodic measures supported on periodic orbits. Our result is applicable
to a wide class of infinite-dimensional dynamical systems. 
\end{abstract}

\maketitle

\section{Introduction}
Let $M$ be a compact metric space and $f\colon M \to M$ a homeomorphism such that $(M, f)$ satisfies the so-called Anosov Closing property, which essentially means that there are many periodic orbits for $f$ in $M$. Furthermore, let $A$ be a linear cocycle over $(M, f)$ that takes values in the space of all bounded linear operators acting on an arbitrary Banach space $\B$. Finally, let $\mu$ be any ergodic $f$-invariant Borel probability measure on $M$. The main objective of the present paper is to show that if $A$ is sufficiently regular (as a map on $M$) and if it satisfies the so-called quasi-compactness property with respect to $\mu$, then all exceptional Lyapunov exponents of $A$ with respect to $\mu$ can be approximated by Lyapunov exponents of $A$ with respect to some ergodic $f$-invariant Borel measure which is supported on a periodic orbit for $f$.

 We emphasize that the assumption that $A$ is quasi-compact with respect to $\mu$ is made to ensure that one can apply the most recent versions of the multiplicative ergodic theorem (MET), which in turn give the set of Lyapunov exponents of $A$ with respect to $\mu$. Consequently, the problem of approximating Lyapunov exponents of $A$ with respect to $\mu$ becomes well-posed. Starting essentially with the pioneering work of Ruelle~\cite{R} who considered cocycles of operators on a Hilbert space, many authors have been interested in the problem of establishing MET for cocycles of operators acting on Banach spaces. In particular, Ma\~n\'e~\cite{M} established MET for cocycles of compact and injective operators on a Banach space. His results were generalized by Thieullen~\cite{Thi87}, who was able to replace the assumption that the operators are compact with a substantially weaker assumption that the cocycle is quasi-compact. More recently, Froyland, Lloyd and Quas~\cite{FLQ10, FLQ13}, Gonz\'alez-Tokman and Quas~\cite{GTQ} and Blumenthal~\cite{AB} were able to remove the assumption present in both~\cite{M} and~\cite{Thi87} (as well as in more recent works such as~\cite{LL}) that the cocycle consists of injective operators. In addition, they have also been able to relax certain regularity assumptions for the cocycle. 
Although the  present paper addresses the problem of the approximation of Lyapunov exponents for quasi-compact cocycles,  we emphasize that our results are new even
 in a particular case of compact cocycles which are not invertible.

In his seminal paper~\cite{Kal11}, Kalinin established (as a tool in proving the main result of~\cite{Kal11}, which is the Liv\v{s}ic theorem) the approximation result described in the first paragraph for cocycles of invertible matrices. This was generalized to cocycles of not necessarily invertible matrices by the first author~\cite{Bac}. Furthermore, Kalinin and Sadovskaya~\cite{KS} (see also~\cite{KS2}) established the approximation result for the largest and smallest Lyapunov exponent of an invertible cocycle acting on an arbitrary Banach space (see Remark~\ref{ksr} for details). 
In the present paper, we go one step further by considering not necessarily invertible cocycles and by establishing the approximation result for all exceptional
Lyapunov exponents and not only for the largest one. 
The importance of our results steems from the fact that in the context of infinite-dimensional dynamics, the invertibility assumption for cocycle is way too restrictive. 
Indeed, the main motivation for papers~\cite{FLQ13, GTQ} was to establish the version of MET that would enable us to study cocycles of transfer operators that are rarely invertible (or even injective). Furthermore, in the recent paper by Blumenthal and Young~\cite{BY} in which the authors extend many results of the smooth ergodic theory to the case of maps acting on Banach spaces, the derivative cocycle is not assumed to be invertible.

The approach and the arguments in the present paper are inspired by those in~\cite{Kal11}. Indeed, when obtaining the approximation property of the largest Lyapunov exponent we follow closely the approach developed in~\cite{Kal11} (which in turn inspired arguments in~\cite{Bac, KS}). However, the nontrivial adaptation of arguments from~\cite{Kal11} occurs when we try to establish the desired approximation property of other Lyapunov exponents. In the classical finite-dimensional case this is done (see~\cite{Kal11, Bac}) by using the so-called exterior powers of the cocycle. On the other hand, such a construction doesn't exist in the infinite-dimensional setting. This forced us to adjust the method of estimating the largest Lyapunov exponent devised in~\cite{Kal11} to fit the problem of estimating other Lyapunov exponents.

The paper is organized as follows. In Section~\ref{sec: statements} we introduce terminology, recall basic notions and important results (such as MET) and state the main
result of our paper. In Section~\ref{LN}, we introduce the concept of Lyapunov norms for operator cocycles which play an important auxiliary tool in our arguments. In Section~\ref{MT} we present the proof of our main result. Finally, in Section~\ref{APP} we discuss various applications of our work in the context of the infinite-dimensional dynamics.

\section{Preliminaries}\label{sec: statements}

Let $(M,d)$ be a compact metric space, $\mu$ a probability  measure defined on the Borel subsets of $M$ and $f: M \to M $ a $\mu$-preserving homeomorphism. Furthermore, 
assume also that $\mu$ is ergodic.

We say that $f$ satisfies the \textit{Anosov Closing property} if there exist $C_1 ,\varepsilon _0 ,\theta >0$ such that if $z\in M$ satisfies $d(f^n(z),z)<\varepsilon _0$ then there exists a periodic point $p\in M$ such that $f^n(p)=p$ and
\begin{displaymath}
d(f^j(z),f^j(p))\leq C_1 e^{-\theta \min\lbrace j, n-j\rbrace}d(f^n(z),z),
\end{displaymath}
for every $j=0,1,\ldots ,n$. We note that shifts of finite type, basic pieces of Axiom A diffeomorphisms and more generally, hyperbolic homeomorphisms are particular examples of maps satisfying the Anosov Closing property. 
We refer to~\cite[Corollary 6.4.17.]{KH95} for details. 

\subsection{Semi-invertible operator cocycles}

Let $(\B,\norm{\cdot})$ be a Banach space and let  $B(\B,\B)$ denote  the space of all bounded linear maps from $\B$ to itself. We recall that $B(\B,\B)$  is  a Banach space with respect to the norm
\[
 \norm{T} =\sup \lbrace \norm{Tv}/\norm{v};\; \norm{v}\neq 0 \rbrace, \quad T\in B(\B,\B).
\]
Although we use the same notation for the norms on $\B$ and $B(\B,\B)$ this will not cause any confusion. 
Finally,  consider a map  $A:M\to B(\B,\B)$.

The \emph{semi-invertible operator cocycle} (or just \textit{cocycle} for short) generated by $A$ over $f$ is  defined as the map $A:\mathbb{N}\times M\to B(\B,\B)$ given by
\begin{equation}\label{def:cocycles}
A^n(x):=A(n, x)=
\left\{
	\begin{array}{ll}
		A(f^{n-1}(x))\ldots A(f(x))A(x)  & \mbox{if } n>0 \\
		\Id & \mbox{if } n=0 \\
	\end{array}
\right.
\end{equation}
for all $x\in M$. 
The term `semi-invertible' refers to the fact that the action of the underlying dynamical system $f$ is assumed to be an invertible transformation while the action on the fibers given by $A$ may fail to be invertible. 

\subsection{Multiplicative ergodic theorem}

We begin  by recalling some terminology. Let $B_\B(0,1)$ denote the unit ball in $\B$ centered at $0$.  For an arbitrary  $T\in B(\B,\B)$, let $\|T\|_\text{ic}$ be the
infimum over all $r>0$ with the property that $T(B_\B(0,1))$ can be covered by finitely many open balls of radius $r$. It is easy to show that:
\begin{equation}\label{ic1}
 \|T\|_\text{ic} \le \lVert T\rVert, \quad \text{for every $T\in B(\B, \B)$}
\end{equation}
and 
\begin{equation}\label{ic2}
 \|T_1 T_2\|_\text{ic} \le \|T_1\|_\text{ic} \cdot \|T_2\|_\text{ic}, \quad \text{for every $T_1, T_2\in B(\B,\B)$.}
\end{equation}
Hence, \eqref{ic2} together with the subadditive ergodic theorem implies that there exists $\kappa(\mu)\in [-\infty, \infty)$ such that
\[
 \kappa(\mu)=\lim_{n\to \infty}\frac 1n \log \lVert A^n(x)\rVert_{ic} \quad \text{for $\mu$-a.e. $x\in M$.}
\]
Observe that if $A$ takes values in a family of compact operators on $\B$, we have that $\kappa(\mu)=-\infty$. Indeed, in this case one has that $\lVert A^n(x)\rVert_{ic}=0$ for each $n$ which readily implies that
$\kappa(\mu)=-\infty$.

In addition, by  using again the subadditive ergodic theorem together with the  subadditivity of the operator norm, we have that there exists $\lambda(\mu)\in [-\infty, \infty)$ such that
\[
 \lambda(\mu)=\lim_{n\to \infty}\frac 1n \log \lVert A^n(x)\rVert \quad \text{for $\mu$-a.e. $x\in M$.}
\]
Note  that~\eqref{ic1} implies that $\kappa(\mu) \le \lambda(\mu)$. We say that the cocycle $A$ is \emph{quasi-compact} with respect to $\mu$ if $\kappa(\mu)<\lambda(\mu)$.
The following result from~\cite[Lemma C.3]{GTQ} gives useful sufficient conditions under which the cocycle is quasi-compact.

\begin{proposition}\label{prop1}
Take $A:M\to B(\B,\B)$.
Let $\mathcal B'=(\mathcal B', \lvert \cdot \rvert)$ be a Banach space such that $\mathcal B\subset \mathcal B'$ and with the property that the inclusion $(\mathcal B, \lVert \cdot \rVert) \hookrightarrow (\mathcal B', \lvert \cdot \rvert)$ is compact. Furthermore, suppose that each $A(x)$ can be extended to a bounded operator on $(\mathcal B', \lvert \cdot \rvert)$ and that there
exist Borel-measurable functions $\alpha, \beta, \gamma \colon M \to (0,\infty)$ such that:
\begin{enumerate}
\item for $\mu$-a.e. $x\in M$ and every $v\in \mathcal B$,
\begin{equation}\label{LY1}
\lVert A(x)v\rVert \le \alpha(x) \lVert v\rVert+\beta(x)\lvert v\rvert;
\end{equation}
\item for $\mu$-a.e. $x\in M$,
\begin{equation}\label{LY2}
\lVert A(x)\rVert \le \gamma(x);
\end{equation}
\item \begin{equation}\label{LY3}
\int \log \alpha \, d\mu < \lambda(\mu) \quad \text{and} \quad \int \log \gamma \, d\mu < \infty.
\end{equation}
\end{enumerate}
Then, $\kappa(\mu) \le \int \log \alpha \, d\mu$. In particular, $A$ is quasi-compact with respect to $\mu$. 
\end{proposition}

\begin{remark}
In the context of cocycles of transfer operators, i.e. when $A(x)$ is the transfer operator associated to some map $T_x$ for each $x\in M$, the condition~\eqref{LY1} is called strong Lasota-Yorke inequality while~\eqref{LY2} is called weak Lasota-Yorke inequality. We note that in this setting one has that $\lambda(\mu)=0$. 

For example, when each $T_x$ is a piecewise expanding map on the unit interval  $[0, 1]$, one can show that under mild assumptions~\eqref{LY1}, \eqref{LY2} and~\eqref{LY3} 
hold with $(\mathcal B, \lVert \cdot \rVert)=(BV, \lVert \cdot \rVert_{BV})$  and $(\mathcal B', \lvert \cdot \rvert)=(L^1, \lVert \cdot \rVert_{L^1})$.
Here, $BV$ denotes the space of all functions of bounded variation on $[0, 1]$ with the corresponding norm $\lVert \cdot \rVert_{BV}$ 
which is defined to be the sum of the $L^1$ norm of the 
function and its total variation.
 We refer to~\cite[Section 2.3.1]{DFGTV} for a detailed discussion. 
\end{remark}

Before stating the version of the multiplicative ergodic theorem established in~\cite{FLQ13}, we recall the notion of $\mu$-continuity.  Let $Z$ be an arbitrary Banach space. We say that a map  $\Phi \colon M \to Z$ is \emph{$\mu$-continuous} if there exists an increasing sequence $(K_n)_{n\in \mathbb N}$ of compact subsets of $M$ satisfying $\mu(\cup_n K_n)=1$ and such that 
$\Phi \rvert_{K_n} \colon K_n \to Z$ is continuous for each $n\in \N$. 
\begin{theorem}\label{thm:Oseledets}
Assume  that the cocycle $A\colon M\to B(\B,\B)$ is $\mu$-continuous and  quasi-compact with respect to $\mu$. Then, we have the following:
 \begin{enumerate}
  \item there exists $l=l(\mu)\in [1, \infty]$ and a sequence of numbers $(\lambda_i(\mu))_{i=1}^l$ such that 
  \[
   \lambda(\mu)=\lambda_1(\mu) >\lambda_2(\mu) >\ldots >\lambda_i(\mu) > \ldots >\kappa(\mu).
  \]
Furthermore, if $l=\infty$ we have that $\lim_{i\to \infty} \lambda_i(\mu) =\kappa(\mu)$;
\item there exists a Borel subset $\Reg \subset M$ such that $\mu(\Reg)=1$ and for each $x\in \Reg$ and $i\in \mathbb N \cap [1, l]$, there is a unique and measurable decomposition
\begin{equation}\label{os}
 \B=\bigoplus_{j=1}^i E_j(x) \oplus V_{i+1}(x),
\end{equation}
where $E_j(x)$ are  finite-dimensional subspaces of $\B$ and $A(x)E_j(x)=E_j(f(x))$. Furthermore, $V_{i+1}(x)$ are closed subspaces of $\B$ and \\ $A(x)V_{i+1}(x)\subset V_{i+1}(f(x))$;
\item for each  $x\in \Reg$  and   $v\in E_j(x)\setminus \{0\}$, we have 
\[
 \lim_{n\to \infty} \frac 1 n \log \lVert A^n(x)v\rVert =\lambda_j(\mu).
\]
In addition, for every $v\in V_{i+1}(x)$,
\[
 \limsup_{n\to \infty} \frac 1 n \log \lVert A^n(x)v\rVert \le \lambda_{i+1}(\mu).
\]

 \end{enumerate}
The numbers $\lambda_i(\mu)$ are called  \emph{exceptional Lyapunov exponents} of the cocycle $A$ with respect to $\mu$ and the dimensions $d_i(\mu)=\dim E_i(x)$ are called \emph{multiplicities} of $\lambda_i(\mu)$. In addition, the decomposition~\eqref{os} is called the \emph{Oseledets splitting}. Finally, the points in $\Reg$ are called \emph{$\mu$-regular} (or simply \emph{regular}).
\end{theorem}

We denote by 
$$\gamma _1(\mu)\geq \gamma _2(\mu)\geq \gamma _3(\mu)\geq \ldots$$
the Lyapunov exponents counted with multiplicities of $A$ with respect to the measure $\mu$. This means that $\gamma _i(\mu)=\lambda_1(\mu)$ for $i=1,\ldots ,d_1(\mu)$, $\gamma _i(\mu)=\lambda_2(\mu)$ for $i=d_1(\mu)+1,\ldots , d_1(\mu)+d_2(\mu)$ and so on. When there is no risk of ambiguity, we suppress the index $\mu$ from the previous objects. Moreover, when the $f$-invariant measure $\mu$ is supported on the orbit of some periodic point $p$ we simply write $\lambda_i(p)$ and $\gamma_i(p)$ for its Lyapunov exponents and Lyapunov exponents counted with multiplicities, respectively. Furthermore, given $x\in M$ and $v\in \B$ we denote by 
$$\lambda(x,v) = \limsup_{n\to\infty} \frac{1}{n}\log \norm{A^n(x)v}$$
the \emph{Lyapunov exponent of $A$ at $x$ in the direction $v$}.

\begin{remark}\label{referee}
Since the arguments in our paper will heavily rely on the  measurability of the  Oseledets splitting~\eqref{os}, we would like to explain what
exactly it means  for~\eqref{os} to be measurable.
Let $\mathcal G(\B)$ denote the set of all closed subspaces $F$ of $\B$ that are complemented, i.e. such that there exists a closed subspace $\tilde F$ of $\B$ with the property that
$\B=F\oplus \tilde F$. We recall that each finite-dimensional subspace $F$ of $\B$ belongs to $\mathcal G(\B)$. It turns out that one can equip $\mathcal G(\B)$ with the 
structure of a  metric space (see~\cite[Section 2.1.2]{BY}) and thus in particular it makes sense to discuss the measurability of the map that is defined on some measurable space
and that takes values in $\mathcal G(\B)$. 

Now we observe that all subspaces of $\B$ that appear in~\eqref{os} belong to $\mathcal G(\B)$. Hence, we can associate to~\eqref{os} the following maps
\begin{equation}\label{maps}
 E_1 \colon \Reg \to \mathcal G(\B), \ldots, E_i \colon \Reg \to \mathcal G(\B) \quad \text{and} \quad V_{i+1} \colon \Reg \to \mathcal G(\B).
\end{equation}
We now say that~\eqref{os} is measurable if all maps in~\eqref{maps} are measurable.  Moreover, those maps are also $\mu$-continuous as a consequence of a deep result by Fremlin~\cite[Theorem 4.1]{KP} (see also~\cite[Remark 3.5.]{BY}). 
\end{remark}

\subsection{Main result}
We say that 
$A:M\to B(\B,\B)$ is  an  \emph{$\alpha$-H\"{o}lder continuous map} if there  exists a constant $C_2>0$ such that
\begin{displaymath}
\norm{A(x)-A(y)} \leq C_2 d(x,y)^{\alpha},
\end{displaymath}
for all $x,y\in M$.   Clearly, if $A:M\to B(\B,\B)$ is  an $\alpha$-H\"{o}lder continuous map, then $A$ is also $\mu$-continuous and consequently Theorem~\ref{thm:Oseledets} is applicable. 
We are now in the position to state the main result of our paper.

\begin{theorem}\label{the: main} 
Let $f: M\to M $ be a homeomorphism satisfying the Anosov Closing property, $\mu$ an ergodic $f$-invariant probability measure and $A:M\to B(\B,\B)$ an $\alpha$-H\"{o}lder continuous map that is quasi-compact
with respect to $\mu$. Then, given $s\in \mathbb{N}\cap [1,l(\mu)]$ there exists a sequence of periodic points $(p_k)_{k\in \mathbb{N}}$ such that
\begin{displaymath}
\gamma_i (p_k)\xrightarrow{k\to +\infty} \gamma _i (\mu) \quad \text{for every $i\in \{1,\ldots ,d_1(\mu)+\ldots+ d_s(\mu)\}$,}
\end{displaymath} 
 where
$d_i(\mu)=\dim E_i(x)$.

\end{theorem}

\begin{remark}
We stress that without the assumption that the cocycle is quasi-compact, 
it is not always possible to get  an approximation result in the spirit of Theorem~\ref{the: main}  even if the cocycle takes values in the space of 
bounded and invertible linear operators on a Banach space. 
Indeed, Kalinin and Sadovskaya \cite[Proposition 1.5]{KS} 
presented an example of a locally constant operator cocycle $A$ over a full shift on two symbols and an ergodic invariant measure $\mu$ such that $\lambda_1(\mu)>\sup_{\mu _p}\lambda_1 (\mu_p)$, where the supremum is taken over all invariant measures $\mu_p$ supported on periodic orbits.
Related examples were also constructed by Hurtado~\cite{Hur}.
\end{remark}

Let us discuss in detail the relationship between Theorem~\ref{the: main} and various related  results in the literature. 
\begin{remark}\label{ksr}
Observe that whenever $\B$ is finite-dimensional and the cocycle is invertible, we have that $\kappa(\mu)=-\infty$ and that 
  the set of exceptional Lyapunov exponents given by Theorem~\ref{thm:Oseledets} coincides with the set of Lyapunov exponents given by
 the classical Oseledets multiplicative ergodic theorem. Therefore, in this setting,  Theorem~\ref{the: main} reduces to~\cite[Theorem 1.4.]{Kal11}.
 
 Recently, the first author~\cite{Bac} has generalized~\cite[Theorem 1.4.]{Kal11} to the case of semi-invertible cocycles of matrices, i.e. $\B$ is again assumed  to
  be finite-dimensional but $A(x)$ doesn't have to be an invertible matrix. In this setting, the family of exceptional Lyapunov exponents can differ from the family of Lyapunov exponents
  given by the version of the Oseledets multiplicative ergodic theorem established in~\cite{FLQ10}. More precisely, let $\Lambda_1$ denote the set of exceptional Lyapunov exponents in the sense
  of Theorem~\ref{thm:Oseledets} and let $\Lambda_2$ denote the set of Lyapunov exponents in the sense of~\cite{FLQ10}. Then, 
  \[
   \Lambda_1=\begin{cases}
          \Lambda_2 & \text{if $-\infty \notin \Lambda_2$;}\\
          \Lambda_2\setminus \{-\infty \}& \text{if $-\infty \in \Lambda_2$.}
             \end{cases}
\]
Since the main result of~\cite{Bac} establishes the 
 desired  approximation property of elements in $\Lambda_2$ including $-\infty$ (if present), we conclude that Theorem~\ref{the: main} provides only a partial generalization of the main result in~\cite{Bac}.

In addition, Kalinin and Sadovskaya~\cite{KS} established the approximation property similar to that  in Theorem~\ref{the: main} for the largest Lyapunov exponent of an arbitrary invertible H\"{o}lder continuous cocycle $A\colon M \to B(\B,\B)$. More precisely, they proved that for each $\epsilon >0$ there exists a periodic point $p\in M$ satisfying
$f^k(p)=p$ and such that 
\[
\bigg{\lvert} \lambda_1 (\mu)-\frac 1 k \log \lVert A^k(q)\rVert \bigg{\rvert} < \epsilon.
\]
However, the above result is weaker than the approximation property for $\lambda_1(\mu)=\gamma_1(\mu)$ established in Theorem~\ref{the: main} (see the discussion in~\cite{KS} after Remark 1.5.).
Moreover, our Theorem~\ref{the: main} deals with \emph{all} exceptional Lyapunov exponents (and not only with the largest one) of a \emph{semi-invertible} quasi-compact cocycle acting on a Banach space and thus represents a  natural extension of the results from~\cite{Kal11, Bac} described above. 

Finally, in their recent paper~\cite{KS2}, Kalinin and Sadovskaya established results similar to those in~\cite{Kal11} and~\cite{KS} for cocycles over  non-uniformly hyperbolic dynamical systems.  Although these systems will in general fail to satisfy the Anosov closing property, they will exhibit a  similar type of behaviour (provided by the so-called Katok's closing lemma).
It turns out that this weaker form of closing property is still sufficient to adapt the arguments from~\cite{Kal11, KS} and obtain the desired approximation property of Lyapunov exponents in this setting. 
\end{remark}

\subsection{Examples}
We now discuss some concrete examples of non-compact cocycles that satisfy all of our assumptions. 

\begin{example}
Assume that $T_1, \ldots, T_k \colon [0, 1] \to [0, 1]$ are piecewise expanding maps such that 
\[
\delta_i:=\essinf_{x\in [0, 1]}\lvert T_i'(x)\rvert >2 \quad \text{for $i\in \{1, \ldots, k\}$.}
\]
Let us denote by $L_i$ the transfer operator associated to $T_i$. We note that $L_i$ acts on the $BV$ space. Furthemore, 
let $M=\{1, \ldots, k\}^{\mathbb Z}$ with the standard topology and consider a two-sided shift $f\colon M \to M$. Note that $(M, f)$ satisfies Anosov closing property. Furthermore, we define a cocycle $A$ on $M$ of operators acting on $BV$ by
\[
A(x)=L_{x_0} \quad \text{for $x=(x_n)_{n\in \mathbb Z} \in M$.}
\]
It is straightforward to verify that $A$ is H\"{o}lder continuous. On the other hand, one can also show (see~\cite[Section 2.3.1]{DFGTV}) that~\eqref{LY1} holds with a  constant $\alpha \in (0, 1)$ and that in fact $A$ is quasi-compact with respect to any $f$-invariant ergodic Borel probability measure. 
\end{example}
The following example is somewhat of different nature. 
\begin{example}
In their recent remarkable paper~\cite{BY},  Blumenthal and Young extend various results from smooth ergodic theory to the case of maps acting on Banach spaces. More precisely, let  $\mathcal B$ be an arbitrary Banach space and consider a $C^2$ Frechet differentiable map $f\colon \mathcal B \to \mathcal B$ with the property
that there exists an compact, $f$-invariant set $\mathcal A \subset \mathcal B$.  In addition, the results in~\cite{BY} assume the existence of an ergodic, $f$-invariant measure $\mu$ such that supp $\mu=\mathcal A$. 

Under the additional assumption that $(\mathcal A, f\lvert_{\mathcal A})$ satisfies Anosov closing property, the results of the present paper  can be used to study the derivative cocycle associated to $f$ which is  given by $A(x)=Df(x)$.
\end{example}

\section{Lyapunov norm}\label{LN}
In order to estimate the growth of the cocycle $A$ along an orbit we introduce the notion of \textit{Lyapunov norm} for 
quasi-compact  operator cocycles  and describe some of its properties. 
This is based on a similar notion introduced in \cite{Bac} in the finite dimensional setting which in turn was based on a similar notion for \emph{invertible cocycles}
that goes back to the work of Pesin (see for instance \cite{BP07}).

\subsection{Lyapunov norm}
Let us use the same notation as in the statement of Theorem~\ref{thm:Oseledets}.  Given $x\in \Reg$, $s\in \mathbb N \cap [1, l(\mu)]$,  $i\in \{1,\ldots , s \}$ and $n\in \mathbb{N}$, 
we consider the map 
$$A^n(f^{-n}(x))_{\mid E_{i}(f^{-n}(x))}: E_{i}(f^{-n}(x))\to E_{i}(x)$$
which is invertible and let us denote its inverse by $\left( A^n(f^{-n}(x))\right)^{-1}_i$. 
Now, for every $n\in \mathbb{Z}$ we can define  the linear map $A^n_i(x):E_{i}(x)\to E_{i}(f^n(x))$  by
\begin{displaymath}
	A^n_i(x)u=\left\{\begin{array}{cc}
		A^n(x)_{\mid E_{i}(x)}u & \mbox{if} \quad n\geq 0\\
		\left( A^{-n}(f^n(x))\right)^{-1}_i u & \mbox{if}\quad n<0.
	\end{array}
	\right.
\end{displaymath}
It is easy to verify (see~\cite[p4.]{Bac}) that 
\begin{equation}
A^{m+n}_i(x)=A^n_i(f^m(x))A^m_i(x), \quad \text{for every $m,n\in \mathbb{Z}$.}
\end{equation}
We are now ready to define the \emph{Lyapunov norm of level s} associated to the operator cocycle $A$ at a regular point $x\in \Reg$: 
 we may write each  $u\in \B$ uniquely as \[u=u_1+\ldots +u_s+u_{s+1}, \] where $u_i \in E_{i}(x)$ for  $i\in \{1, \ldots, s\}$ and $u_{s+1}\in V_{s+1}(x) $. Thus, given $\delta >0$ we define its \textit{$\delta$-Lyapunov norm of level s} by
\begin{displaymath}
\norm{u}_{x,\delta}=\sum _{i=1}^{s+1}\norm{u_i}_{x,\delta,i},
\end{displaymath}
where
\begin{equation}\label{eq: def Lyap i norm}
\norm{u_i}_{x,\delta,i}= \sum _{n\in \mathbb{Z}}\norm{A^{n}_i(x)u_i} e^{-\lambda _i n -\delta \mid n\mid} \quad \text{ $i\in \{1, \ldots, s\}$}
\end{equation}
 and 
\begin{equation}\label{eq: def Lyap s+1 norm}
\norm{u_{s+1}}_{x,\delta,s+1}= \sum _{n=0}^{+\infty}\norm{A^{n}(x)\tilde{u}}e^{-\tilde{\lambda} n}.
\end{equation}
Here $\tilde{\lambda}$ is any fixed number smaller than $\lambda_s(\mu)$ with the property that $[\tilde{\lambda},\lambda_s(\mu))\cap \Lambda(\mu)=\emptyset$, where $\Lambda(\mu)$ 
denotes the set of all exceptional Lyapunov exponents of $A$ with respect to $\mu$.  
Observe that such number $\tilde{\lambda}$ does exist since  by Theorem~\ref{thm:Oseledets} elements of $\Lambda(\mu)$ can only accumulate at $\kappa(\mu)$. 
Moreover, both series \eqref{eq: def Lyap i norm} and \eqref{eq: def Lyap s+1 norm} converge.  Indeed, this follows readily from the following lemma.

\begin{lemma}\label{lem: auxil 1}
For every $u\in E_{i}(x)\setminus \{0\}$,
$$\lim _{n\to \pm \infty}\frac{1}{n}\log \norm{A^n_i(x)u}=\lambda _i.$$
Moreover, there exists $\epsilon> 0$ such that for every $\tilde{u}\in V_{s+1}(x)$,
$$\limsup _{n\to + \infty}\frac{1}{n}\log \norm{A^n(x)\tilde{u}}<\tilde{\lambda}-\epsilon.$$
\end{lemma}
\begin{proof} 
The first assertion is a consequence of \cite[Lemma 20.]{FLQ13}, while the second claim follows easily from the choice of $\tilde{\lambda}$ and the properties of the Oseledets splitting
given by Theorem~\ref{thm:Oseledets}.
\end{proof}

One can easily verify that   $\norm{\cdot}_{x,\delta}$ is indeed a norm on $\B$. When there is no risk of ambiguity, we will  write $\norm{\cdot}_x$ and $\norm{\cdot}_{x,i}$ 
instead of $\norm{\cdot}_{x,\delta}$ and $\norm{\cdot}_{x,\delta,i}$ respectively, and call it simply \emph{Lyapunov norm}.

Given a bounded linear operator $T\in B(\B,\B)$, its Lyapunov norm with respect to $x,y\in \Reg$ is defined by
\begin{displaymath}
\norm{T}_{y\leftarrow x}=\sup \{\norm{Tu}_{y}/\norm{u}_{x}; \; u\in \B\setminus \{0\} \}.
\end{displaymath}

\subsection{Auxiliary result} In the next section we are going to describe some properties of the Lyapunov norm. 
In order to do so, we need the following auxiliary result which is a version of Theorem 2 from \cite{DrF} for cocycles acting on Banach spaces. 

\begin{proposition}\label{lem: version of theo 2} 
Given $x\in \Reg$, let us consider the splitting 
$$\B=E_1(x)\oplus\ldots \oplus E_{s}(x)\oplus V_{s+1}(x).$$
There exists a full $\mu$-measure set $\Omega\subset \Reg$ so that for each $\varepsilon >0$ small enough there are function $C, K:M\to (0,+\infty)$ satisfying for every $x\in \Omega$, 
the following properties:
\begin{itemize}
\item[i)] for each $1\leq i\leq s$, $u\in E_i(x)$ and $n\in \mathbb{Z}$,
 $$\frac{1}{C(x)}e^{\lambda _i n-\varepsilon \mid n\mid }\norm{u}\leq \norm{A_i^n(x)u}\leq C(x)e^{\lambda _i n+\varepsilon \mid n\mid}\norm{u};$$
 \item[ii)] for each $\tilde{u}\in V_{s+1}(x)$ and $n\in \mathbb{N}$,
 $$ \norm{A^n(x)u}\leq C(x)e^{(\tilde{\lambda}-\varepsilon )n}\norm{u};$$
 \item[iii)] $C(f^n(x))\leq C(x)e^{\varepsilon \mid n\mid}$ for every $n\in \mathbb{Z}$.
 \item[iv)] $K(f^n(x))\leq K(x)e^{\varepsilon \mid n\mid}$ for every $n\in \mathbb{Z}$ and
 \[
  \lVert u\rVert \le K(x)\lVert u+v\rVert \quad \text{and} \quad \lVert v\rVert \le K(x)\lVert u+v\rVert,
 \]
for $u\in E_1(x)\oplus\ldots \oplus E_{s}(x)$ and $v\in V_{s+1}(x)$.
\end{itemize}
\end{proposition} 

We will use the following well-known result (see~\cite{BY}).

\begin{theorem}[John's Theorem]\label{johnthm}
 Let $E\subset \mathcal B$ be a subspace of dimension $k\in \mathbb N$. Then, there exists a scalar product $\langle \cdot, \cdot \rangle_E$ on $E$ that induces norm $\lVert \cdot \rVert_E$
 such that
 \begin{equation}\label{ne}
  \lVert v\rVert_E \le \lVert v\rVert \le \sqrt k \lVert v\rVert_E, \quad \text{for each $v\in E$.}
 \end{equation}

\end{theorem}

\begin{proof}[Proof of Proposition~~\ref{lem: version of theo 2}]
 We follow closely the arguments in~\cite{BY, DrF}.
Take any $i\in \{1, \ldots, s\}$. 
\begin{lemma}\label{9216}
We have
\begin{equation}\label{res}
\limsup_{n\to \infty} \frac {1}{n} \log \lVert A_i^n(x)\rVert \le \lambda_i \quad \text{for $\mu$-a.e. $x\in X$.}
\end{equation}
\end{lemma}

\begin{proof}[Proof of the lemma]
Let $\langle \cdot, \cdot \rangle_{E_i(x)}$ be a scalar product on $E_i(x)$ given by Theorem~\ref{johnthm} and let $\lVert \cdot \rVert_{E_i(x)}$ denote the associated norm.
Let $\{e_1, \ldots, e_t\}$ be an orthonormal basis for $E_i(x)$, $t=\dim E_i(x)$.  For each $n\in \N$, choose $v_n\in E_i(x)$  such that
$\lVert v_n\rVert=1$ and $ \lVert A_i^n(x)\rVert =\lVert  A^n (x)v_n\rVert$. Furthermore, for  $n\in \N$, write $v_n$ in the form
\[
v_n=\sum_{j=1}^t a_{j,n}e_j, \quad \text{for some $a_{j,n}\in \R$. }
\]
 We note that it follows from~\eqref{ne} that 
 \[\lvert a_{j,n}\rvert =\lvert \langle v_n, e_j\rangle_{E_i(x)} \rvert \le \lVert v_n\rVert_{E_i(x)} \cdot \lVert e_j\rVert_{E_i(x)}\le 1\] and thus
\begin{equation}\label{TAA}
\lVert A_i^n(x)\rVert \le \sum_{j=1}^t \lvert a_{j,n}\rvert \cdot \lVert A^n(x)e_j\rVert \le \sum_{j=1}^t  \lVert A^n(x)e_j\rVert.
\end{equation}
Since $e_j \in E_i(x)$, 
\begin{equation}\label{TBB}
\lim_{n\to \infty} \frac 1n \log \lVert A^n(x)e_j\rVert = \lambda_i, \quad \text{for $j\in \{1, \ldots, t\}$.}
\end{equation}
It remains to observe that~\eqref{TAA} and~\eqref{TBB} readily imply~\eqref{res}.

\end{proof}
It follows from~\eqref{res} that for $\epsilon >0$,
\begin{equation}\label{defD}
D(x):=\sup_{n\ge 0} \{\lVert A_i^n(x)\rVert  \cdot e^{-(\lambda_i+\epsilon)n}\} <\infty,
\end{equation}
for $\mu$ a.e.\ $x\in X$.
\begin{lemma}
We have
\begin{equation}\label{temp}
\lim_{n \to \pm \infty}\frac{1}{n} \log D(f^n(x))=0 \quad \text{for $\mu$-a.e. $x\in X$.}
\end{equation}
\end{lemma}
\begin{proof}[Proof of the lemma]
For $n\ge 1$, we have
\[
\begin{split}
\lVert A_i^n(x)\rVert  &\le \lVert A_i^{n-1}(f(x))\rVert \cdot \lVert A_i(x)\rVert  \\
&\le \lVert A_i^{n-1}(f(x))\rVert \cdot \lVert A(x)\rVert.
\end{split}
\]
By multiplying the above inequality by $e^{-(\lambda_i+\epsilon)n}$, we obtain
\[
e^{-(\lambda_i+\epsilon)n}\lVert A_i^n(x)\rVert \le e^{-(\lambda_i+\epsilon)(n-1)}\lVert A_i^{n-1}(f(x))\rVert \cdot e^{-(\lambda_i+\epsilon)}\lVert A(x)\rVert.
\]
Hence,
\[
D(x) \le D(f(x))\cdot \max \{e^{-(\lambda_i+\epsilon)}\lVert A(x)\rVert, 1\}.
\]
It follows from the continuity of $A$ and compactness of $M$ that there exists $T>0$ such that
\begin{equation}\label{X}
\log D(x)- \log D(f(x))\le T.
\end{equation}
Set
\[
\tilde D(x)=\log D(x)- \log D(f(x)).
\]
We note that
\begin{equation}\label{Z1}
\frac{1}{n}\log D(f^n(x))=\frac{1}{n}\log D(x)-\frac{1}{n}\sum_{j=0}^{n-1}\tilde D(f^j(x)),
\end{equation}
for each $x\in X$ and $n\in \N$. Hence, we can apply the  Birkhoff ergodic theorem  and conclude that there exists $a\in [-\infty, \infty)$ such that
\begin{equation}\label{Z2}
\lim_{n\to \infty} \frac{1}{n} \sum_{j=0}^{n-1}\tilde D(f^j(x))=a,
\end{equation}
for $\mu$-a.e. $x\in X$. It follows from~\eqref{Z1} and~\eqref{Z2} that
\[
\lim_{n\to \infty} \frac{1}{n}\log D(f^n(x))=-a.
\]
On the other hand, since $\mu$ is $f$-invariant, for any $c>0$ we have that
\[
\lim_{n\to \infty} \mu (\{x\in X: \log D(f^n(x))/n \ge c\})=\lim_{n\to \infty} \mu (\{x\in X: \log D(x) \ge nc\})=0,
\]
which immediately implies that $a\ge 0$. Thus,
\[
\lim_{n\to \infty} \frac{1}{n} \log D(f^n(x))\le 0.
\]
It follows from~\eqref{defD} that $D(x)\ge 1$ for $\mu$ a.e.\ $x\in X$ and therefore we can conclude that~\eqref{temp} holds when $n\to \infty$. 
One can similarly establish~\eqref{temp} for the case $n\to -\infty$. 

\end{proof}
It follows from~\eqref{temp} and~\cite[Proposition 4.3.3(ii)]{A} that there exists a nonnegative and measurable function $C$ defined on a set of full-measure  satisfying 
inequality in part iii) in the statement of the Lemma and 
such that $D(x)\le C(x)$, which together with~\eqref{defD} implies that the second inequality in the part i) of the Lemma holds. The proof of ii) is analogous. 
\begin{lemma}\label{inv_integ}
We have
\[
\int_X \log^+ \lVert A_i(x)^{-1}\rVert \, d\mu(x) < \infty.
\]
\end{lemma}

\begin{proof}[Proof of the lemma]
One can repeat arguments from \cite[Lemma 4]{DrF} using $\lVert \cdot \rVert_{E_i(x)}$ from Lemma~\ref{9216} instead of the original norm to establish that
\[
 \int_X \log^+ \lVert A_i(x)^{-1}\rVert' \, d\mu(x) < \infty,
\]
where 
\[
 \lVert A_i(x)^{-1}\rVert' =\sup_{\lVert v\rVert_{E_i(x)} \le 1} \lVert  A_i(x)^{-1}v\rVert_{E_i(x)}.
\]
In  view of~\eqref{ne}, the conclusion of the lemma follows. 
\end{proof}
We now prove that the first inequality in i) holds. Let us consider the cocycle $x\mapsto B(x):= A_i(f^{-1}(x))^{-1}$ over $f^{-1}$ that acts on 
 a subbundle $E_i(x)$.  Then, $-\lambda_i$ is the only
Lyapunov exponent of $B$. Because of Lemma~\ref{inv_integ}, we can apply the first part of the proof to $B$  to  conclude that that there exists a function $C\colon M \to (0, \infty)$ such that
\begin{equation}\label{Y1}
\lVert B^n(x) v\rVert \le C(x)e^{(-\lambda_i +\frac{\epsilon}{2})n}, \quad \text{for $\mu$-a.e. $x\in M$, $n\ge 0$ and $v\in E_i(x)$}
\end{equation}
and
\begin{equation}\label{Y2}
C(f^m(x))\le C(x)e^{\frac{\epsilon}{2}  \lvert m\rvert}, \quad \text{for $\mu$-a.e. $x\in M$ and $m\in \mathbb Z$,}
\end{equation}
which readily implies first estimate in i). Finally, the existence of a function $K$ that satisfies assertion iv) follows  from~\cite[Lemma 1.]{DrF}. The proof of Proposition~\ref{lem: version of theo 2} is completed. 
\end{proof}

\subsection{Properties of the Lyapunov norm} Some useful properties of the Lyapunov norm are given in the next proposition. 

\begin{proposition}\label{prop: properties of the lyapunov norm} Let $x\in \Reg$.

i) For every $1\leq i\leq s$, $u\in E_{i}(x)$ and $n\in \N$, we have that 
\begin{equation} \label{ineq: Lyapunov norm}
e^{(\lambda _i -\delta)n}\norm{u}_{x,i}\leq \norm{A^n(x)u}_{f^n(x),i}\leq e^{(\lambda _i +\delta)n}\norm{u}_{x,i};
\end{equation}

ii) For every $u\in V_{s+1}(x)$ and $n\in \mathbb{N}$, we have that
$$\norm{A^n(x)u}_{f^n(x),s+1}\leq e^{\tilde{\lambda} n}\norm{u}_{x,s+1}; $$

iii) For every $\delta >0$ and $n\in \N$,  we have that
\begin{equation}\label{ineq: norm}
 \norm{A^n(x)}_{f^n(x)\leftarrow x}\leq e^{(\lambda _1 +\delta)n};
\end{equation}

iv) For every $\delta >0$ sufficiently small,  there exists a measurable function $K_{\delta}:\Reg \to (0,+\infty)$ such that
\begin{equation}\label{ineq: norm x Lyapunov norm}
\norm{u}\leq \norm{u}_x\leq K_{\delta}(x)\norm{u} \quad \text{for $x\in \Reg$ and $u\in \B$.}
\end{equation}
Furthermore, 
\begin{equation}\label{ineq: K delta growth}
K_{\delta}(x)e^{-\delta n}\leq K_{\delta}(f^n(x))\leq K_{\delta}(x)e^{\delta n} \quad \text{for $x\in \Reg$ and $n\in \N$.}
\end{equation} 
Consequently, for any  $B\in \mathcal B(\B, \B)$ and any two regular points $x$ and $y$, we have that
\begin{equation}\label{ineq: norm x Lyapunov norm operator}
K_{\delta}(x)^{-1}\norm{B}\leq \norm{B}_{y\leftarrow x}\leq K_{\delta}(y)\norm{B}.
\end{equation}

\end{proposition}

\begin{proof}
In order to prove $i)$ we observe that for any $u\in E_{i}(x)$,
\begin{displaymath}
\begin{split}
\norm{A(x)u}_{f(x),i}&= \sum _{n\in \mathbb{Z}}\norm{A^{n}_i(f(x))A(x)u} e^{-\lambda _i n -\delta \mid n\mid}\\
&= \sum _{n\in \mathbb{Z}}\norm{A^{n+1}_i(x)u} e^{-\lambda _i n -\delta \mid n\mid}\\
&= \sum _{n\in \mathbb{Z}}\norm{A^{n+1}_i(x)u} e^{-\lambda _i (n+1) -\delta \mid n+1\mid} e^{\lambda _i+ \delta (\mid n+1\mid -\mid n\mid )}.
\end{split}
\end{displaymath}
Consequently,
$$e^{(\lambda _i -\delta)}\norm{u}_{x,i}\leq \norm{A(x)u}_{f(x),i}\leq e^{(\lambda _i +\delta)}\norm{u}_{x,i}, $$
which readily implies $i)$. The proof of item $ii)$ is analogous. Indeed, we have that 
\begin{displaymath}
\begin{split}
\norm{A^n(x)u}_{f^n(x),s+1}&= \sum _{k=0}^{+\infty}\norm{A^{k}(f^n(x))A^n(x)u} e^{-\tilde{\lambda}k}\\
&= \sum _{k=0}^{+\infty}\norm{A^{k+n}(x)u} e^{-\tilde{\lambda}(k+n)}e^{\tilde{\lambda}n}\leq e^{\tilde{\lambda}n}\norm{u}_{x,s+1},\\
\end{split}
\end{displaymath}
for each $u\in V_{s+1}(x)$.

In order to obtain $iii)$, take an arbitrary $u\in \B$ and write it  in the form
\begin{equation}\label{equ} u=u_1+\ldots +u_s+u_{s+1},\end{equation} where $u_i \in E_i(x)$ for $i\in \{1, \ldots, s\}$ and $u_{s+1} \in V_{s+1}(x)$. Then, it follows from~$i)$ and $ii)$ that 
\begin{displaymath}
\begin{split}
\norm{A^n(x)u}_{f^n(x)}&=\sum _{i=1}^{s+1}\norm{A^n(x)u_i}_{f^n(x),i}\\
&\leq   \sum _{i=1}^{s}e^{(\lambda _i+\delta)n}\norm{u_i}_{x,i} +e^{\tilde{\lambda} n}\norm{u_{s+1}}_{x,s+1}\\
& \leq  e^{(\lambda _1+\delta)n}  \sum _{i=1}^{s+1}\norm{u_i}_{x,i}=e^{(\lambda _1+\delta)n} \norm{u}_{x},
\end{split}
\end{displaymath}
which implies the desired conclusion.

The first inequality of $iv)$ is trivial. In order to prove the second one, take $\varepsilon\in (0, \frac{\delta}{2})$  small enough and 
let $C:\Reg \to (0,\infty)$ be the map given by Proposition \ref{lem: version of theo 2} (diminishing $\Reg$, if necessary, we may assume $\Omega=\Reg$). 
Thus, for every $1\leq i\leq s$, $u\in E_i(x)$ and $n\in \mathbb{Z}$, we have
 $$\frac{1}{C(x)}e^{\lambda _i n-\varepsilon \mid n\mid }\norm{u}\leq \norm{A_i^n(x)u}\leq C(x)e^{\lambda _i n+\varepsilon \mid n\mid}\norm{u}.$$
Therefore, 
\begin{equation}\label{0752}
\begin{split}
\norm{u}_{x,i}&=\sum _{n\in \mathbb{Z}}\norm{A^n_i(x)u}e^{-\lambda_i n -\delta \mid n\mid}\\
&\leq \sum _{n\in \mathbb{Z}}\left( C(x)e^{\lambda_i n + \varepsilon \mid n\mid}\norm{u}\right)e^{-\lambda_i n -\delta \mid n\mid}\\
&=C(x)\sum _{n\in \mathbb{Z}}e^{ (\varepsilon -\delta )\mid n\mid} \norm{u}.
\end{split}
\end{equation}
On the other hand, for $u\in V_{s+1}(x)$,  Proposition \ref{lem: version of theo 2} implies that 
$$ \norm{A^n(x)u}\leq C(x)e^{(\tilde{\lambda}-\varepsilon )n}\norm{u},$$
for each $n\in \mathbb{N}$. Thus, 
\begin{equation}\label{0753}\norm{u}_{x,s+1}= \sum _{n\geq 0}\norm{A^n(x)u}e^{-\tilde{\lambda} n}\leq C(x)\sum _{n\geq 0}e^{-\varepsilon n}\norm{u}.\end{equation}
Set \[K=\max \bigg{\{}\sum _{n\in \mathbb{Z}}e^{ (\varepsilon -\delta )\mid n\mid}, \sum _{n\geq 0}e^{-\varepsilon n} \bigg{\}}.\] 
Take now an arbitrary $u\in \B$ and write it in the form~\eqref{equ}, where 
 $u_i\in E_{i}(x)$ for $i\in \{1, \ldots,  s\}$ and $u_{s+1} \in V_{s+1}(x)$. Then, it follows from~\eqref{0752} and~\eqref{0753} that 
$$\norm{u}_x=\sum_{i=1}^{s+1}\norm{u_i}_{x,i}\leq KC(x)\sum _{i=1}^{s+1}\norm{u_i}.$$
It remains to obtain an upper bound for $\lVert u_i\rVert$ in terms of $\lVert u\rVert$. This can be achieved by using the map $K$ given by Proposition~\ref{lem: version of theo 2}. More precisely, let
$K^1$ be the map given by Proposition~\ref{lem: version of theo 2} applied for $s=1$ and sufficiently small $\epsilon >0$. We then have that
\begin{equation}\label{new1}
 \lVert u_1\rVert \le K^1(x)\lVert u\rVert \quad \text{and} \quad \lVert u_2+\ldots +u_{s+1}\rVert \le K^1(x)\lVert u\rVert 
\end{equation}
The first inequality in~\eqref{new1} gives a desired bound for $\lVert u_1\rVert$. In order to obtain the bound for $\lVert u_2\rVert$, we can apply again Proposition~\ref{lem: version of theo 2}
but now for $s=2$ (and again for $\epsilon >0$ sufficiently small) to conclude that there exists $K^2$ such that
\begin{equation}\label{new2}
 \lVert u_2\rVert \le K^2(x)\lVert u_2+\ldots +u_{s+1}\rVert \quad \text{and} \quad \lVert u_3+\ldots +u_{s+1}\rVert \le K^2(x)\lVert u_2+\ldots +u_{s+1}\rVert.
\end{equation}
By combining the second inequality in~\eqref{new1} with the first inequality in~\eqref{new2}, we conclude that $\lVert u_2\rVert \le K^1(x)K^2(x)\lVert u\rVert$. By proceeding, one can
establish desired bounds for all $\lVert u_j\rVert$, $j=1, \ldots, s+1$ and construct function $K_\delta$.

\end{proof}

For any $N>0$, let $\Reg _{\delta ,N}$ be the set of regular points $x\in \Reg$ for which $K_{\delta}(x)\leq N$. 
Observe that $\mu(\Reg _{\delta, N})\to 1$ as $N\to +\infty$. 
Moreover, invoking Lusin's theorem together with the $\mu$-continuity of decomposition~\eqref{os} for $i=s$ (see Remark~\ref{referee}), we may assume without loss of generality that this set is compact and that the Lyapunov norm and the 
Oseledets splitting are continuous when restricted to it.

\section{Proof of Theorem \ref{the: main}}\label{MT}

Let $f: M\to M $, $A:M\to B(\B,\B)$, $\mu$ and $s\in \mathbb{N}\cap [1,l(\mu)]$ be given as in the statement of Theorem~\ref{the: main}. 
We may assume without loss of generality that $\mu$ is not supported on a periodic orbit since otherwise there is nothing to prove. Recall that $d_i(\mu)=\text{dim}(E_i(x))$ and consider $d=10\prod_{i=1}^{s} (d_i(\mu)+4)$. Take $\delta_0 >0 $ so that $\delta _0 < \frac{1}{d}\min _{i=1,\ldots ,s} \{\theta \alpha , (\lambda _i -\lambda _{i+1}) \}$ if $l(\mu)\geq 2$ and  $\delta _0 <\frac{1}{4}\theta \alpha$ otherwise. 
Fix $N>0$ and $\delta \in (0,\delta _0)$. 

Let 
\begin{displaymath}
	B(\mu)=\left\{ x\in M; \; \dfrac{1}{n}\sum _{i=0}^{n-1}\delta _{f^i(x)}\xrightarrow{n\to \infty} \mu \quad \mbox{in the weak$^{\ast}$ topology} \right\}
 \end{displaymath}
be the \emph{basin} of $\mu$. Since $\mu$ is ergodic, $B(\mu)$ has full measure. 
Choose  $x\in B(\mu)\cap \Reg _{\delta ,N}$ such  that $\mu (B(x,\frac{1}{k})\cap \Reg _{\delta ,N})>0$ for every $k\in \mathbb{N}$, where 
$B(x,\frac{1}{k})$ denotes the open ball of radius $\frac{1}{k}$ centered at $x$. 
By Poincar\'e's Recurrence Theorem,  there exists a sequence $(n_k)_{k\in \mathbb{N}}$ of positive integers so that $n_k\to +\infty$ 
and $f^{n_k}(x)\in B(x,\frac{1}{k})\cap \Reg _{\delta ,N}$ for each $k\in \mathbb{N}$. By the Anosov Closing property it follows that, for each $k$ sufficiently large,
there exists a periodic point $p_k$ of period $n_k$ such that
\begin{equation}\label{eq: Anosov closing 2}
d(f^j(x),f^j(p_k))\leq C_1 e^{-\theta \min\lbrace j, n_k-j\rbrace}d(f^{n _k}(x),x)\leq \frac{C_1}{k} e^{-\theta \min\lbrace j, n_k-j\rbrace},
\end{equation}
for every $j\in \{0,1,\ldots , n_k\}$. For each $k\in \mathbb{N}$, let us consider the ergodic \emph{periodic measure} given by 
$$\mu _{p_k} =\dfrac{1}{n_k}\sum _{j=0}^{n_k-1}\delta _{f^j(p_k)}.$$
From the choice of $x\in B(\mu)$ and  \eqref{eq: Anosov closing 2} it follows that the sequence $\{\mu_{p_k}\}_{k\in \mathbb{N}}$ converges to $\mu$ in the weak$^*$-topology.

In order to simplify the proof, we will split it into several lemmas. 

\begin{lemma}\label{lem: upper semic}The map 
$$\mu\to \gamma _1(\mu)+\gamma_2(\mu)+\ldots +\gamma_i(\mu)$$
is upper-semicontinuous for every $i\in \{1,\ldots ,s\}$.
\end{lemma}

\begin{proof}
Let us fix $i\in \{1, \ldots, s\}$. 
It follows from~\cite[Lemma A.3]{DFGTV} that there exists a subadditive sequence $(F_n)_{n\ge 1}$ of  functions $F_n \colon M \to \mathbb R$ such that 
\[
 \gamma _1(\mu)+\gamma_2(\mu)+\ldots +\gamma_i(\mu)=\inf_{n\in \mathbb N} \frac 1 n \int_M F_n(q)\, d\mu(q).
\]
The desired conclusion can now be obtained by using standard arguments as in~\cite[Lemma 9.1]{LLE}.
\end{proof}

The following is a simple consequence of Lemma~\ref{lem: upper semic}.

\begin{corollary}\label{cor: upper semic p_k}
We have that 
$$\limsup _{k\to +\infty}\left(\gamma _1(p_k)+\gamma_2(p_k)+\ldots +\gamma_i(p_k)\right) \leq \gamma _1(\mu)+\gamma_2(\mu)+\ldots +\gamma_i(\mu), $$
for every $i\in \{1,\ldots ,d_1(\mu)+\ldots +d_s(\mu)\}$.
\end{corollary}

\subsection{Approximation of the largest Lyapunov exponent}\label{sec: largest LP}

For each $1\leq j\leq n_k$, let us consider the splitting  \[ \B=E_{1}(f^j(x))\oplus V_2(f^j(x)) \]
and write $u\in \B$ as $u=u^j_E+u^j_V$,  where $u^j_E\in E_{1}(f^j(x))$ and $u^j_V\in V_2(f^j(x))$. Then the \textit{cone} of radius $1-\gamma > 0$ around $E_{1}(f^j(x))$ is defined as  
\begin{displaymath}
C^{j,1}_{\gamma }=\left\{ u^j_E+u^j_V\in E_{1}(f^j(x))\oplus V_2(f^j(x)); \; \norm{u^j_V}_{f^j(x)}\leq (1-\gamma) \norm{u^j_E}_{f^j(x)}\right\}.
\end{displaymath}
To simplify notation we write $\norm{\cdot}_j$ for the Lyapunov norm at the point $f^j(x)$.

\begin{lemma} \label{lem: main 1}
For every $1\leq j\leq n_k$ and $u\in C^{j,1}_0$,
\begin{equation}\label{ineq: Lyap norm cone}
\norm{(A(f^j(p_k))u)^{j+1}_E}_{j+1}\geq e^{\lambda _1 -2\delta}\norm{u^j_E}_j.
\end{equation}
Moreover, for $k$ sufficiently large there exists $\gamma \in (0,1)$ such that 
\begin{equation}\label{ineq: cone}
A(f^j(p_k))(C^{j,1}_0)\subset C^{j+1,1}_{\gamma}.
\end{equation}
\end{lemma}

\begin{proof} 
Given $u\in C^{j,1}_0$ let us consider $v=A(f^j(x))u$. Then, it follows from \eqref{ineq: Lyapunov norm} that $\norm{v}_{j+1}\leq e^{\lambda _1 +\delta}\norm{u}_j$ and moreover that 
\begin{displaymath}
\norm{v^{j+1}_E}_{j+1}=\norm{A(f^j(x))u^j_E}_{j+1}\geq e^{\lambda _1 -\delta}\norm{u^j_E}_j
\end{displaymath}
and
\begin{equation}\label{eq: norm vjF}
\norm{v^{j+1}_V}_{j+1}=\norm{A(f^j(x))u^j_V}_{j+1}\leq e^{\lambda _2 +\delta}\norm{u^j_V}_j.
\end{equation}
Let $w=A(f^j(p_k))u$. We now wish  to compare the Lyapunov norms of $w$ and its projection onto $E_{1}(f^{j+1}(x))$ and $V_2(f^{j+1}(x))$ with the respective norms of $v$. 
Set  $B_j=A(f^j(p_k))-A(f^j(x))$. Consequently, $w=v+B_ju$ and thus \[w^{j+1}_E=v^{j+1}_E+ (B_ju)^{j+1}_E \quad \text{and} \quad  w^{j+1}_V=v^{j+1}_V+ (B_ju)^{j+1}_V.\] Moreover, we have 
\begin{displaymath}
\begin{split}
\norm{B_j}&=\norm{A(f^j(p_k))-A(f^j(x))}\leq C_2 d(f^j(p_k),f^j(x))^{\alpha} \\
& \leq C_1^\alpha C_2 \frac{1}{k^{\alpha}}e^{-\theta \alpha \min\{ j, n_k-j\}},
\end{split}
\end{displaymath}
for every $0\leq j\leq n_k$.
Therefore, invoking  \eqref{ineq: norm x Lyapunov norm} and \eqref{ineq: norm x Lyapunov norm operator}  
it follows that
\begin{displaymath}
\norm{B_ju}_{j+1}  \leq \norm{B_j}_{f^{j+1}(x)\leftarrow f^{j+1}(x)}\norm{u}_{j+1}\leq K_{\delta}(f^{j+1}(x))^2\norm{B_j}\norm{u}.
\end{displaymath}
Since $x$ and $f^{n_k}(x)$ belong to $\Reg _{\delta,N}$, it follows from \eqref{ineq: K delta growth} that 
\[ K_{\delta}(f^{j+1}(x)) \leq N e^{\delta \min\{ j+1, n_k-j-1\}}.\] 
The above inequality together with 
$\norm{u}_j\leq 2\norm{u_E^j}_j$ (recall that $u\in C^{j,1}_0$) implies that 
\begin{displaymath}
\begin{split}
\norm{B_ju}_{j+1} &\leq  N^2 e^{2\delta\min\{ j+1, n_k-j-1\}} C_1^\alpha C_2 \frac{1}{k^{\alpha}}e^{-\theta \alpha \min\{ j, n_k-j\}}\norm{u}_j\\
&\leq  C_1^\alpha C_2 N^2 \frac{1}{k^{\alpha}} e^{2\delta\min\{ j+1, n_k-j-1\}} e^{-\theta \alpha \min\{ j, n_k-j\}}2\norm{u^j_E}_j\\
& \leq C \frac{1}{k^{\alpha}}e^{(2\delta -\theta \alpha) \min\{ j, n_k-j\}}\norm{u^j_E}_j,
\end{split}
\end{displaymath}
where  $C:= 2C_1^\alpha C_2 N^2>0$.
Thus, since $2\delta -\theta \alpha <0$, we obtain that \[\norm{B_ju}_{j+1}\leq \tilde{C}\frac{1}{k^{\alpha}}\norm{u^j_E}_j,\] for some $\tilde{C}>0$ independent of $n_k$ and $j$. Consequently,
\begin{displaymath}
\begin{split}
\norm{w^{j+1}_E}_{j+1} & \geq \norm{v^{j+1}_E}_{j+1} -\norm{(B_ju)^{j+1}_E}_{j+1}\\
&\geq e^{\lambda _1 -\delta}\norm{u^j_E}_j -\tilde{C}\frac{1}{k^{\alpha}}\norm{u^j_E}_j \\
&\geq  e^{\lambda _1 -2\delta}\norm{u^j_E}_j,
\end{split}
\end{displaymath}
whenever $k$ is sufficiently large which is precisely the inequality \eqref{ineq: Lyap norm cone}. 

In order to obtain~\eqref{ineq: cone}, we observe initially that 
\begin{equation}\label{aux: lower 1.0}
\norm{w^{j+1}_E}_{j+1}  \leq  e^{\lambda _1 +\delta}\norm{u^j_E}_j +\tilde{C}\frac{1}{k^{\alpha}}\norm{u^j_E}_j \leq  \hat{C}\norm{u^j_E}_j.
\end{equation}
On the other hand,
\begin{displaymath}
\norm{w^{j+1}_E}_{j+1}  \geq \norm{v^{j+1}_E}_{j+1}  -\norm{B_ju}_{j+1}
\end{displaymath}
and
\begin{displaymath}
\norm{w^{j+1}_V}_{j+1}  \leq \norm{v^{j+1}_V}_{j+1}  +\norm{B_ju}_{j+1}. 
\end{displaymath}
Therefore, combining these inequalities and using again that $u\in C^{j,1}_0$, we have that 
\begin{displaymath}
\begin{split}
\norm{w^{j+1}_E}_{j+1} - \norm{w^{j+1}_V}_{j+1} & \geq \norm{v^{j+1}_E}_{j+1} -\norm{v^{j+1}_V}_{j+1}  -2\norm{B_ju}_{j+1}\\
&\geq e^{\lambda _1-\delta}\norm{u^{j}_E}_{j} - e^{\lambda _2+\delta}\norm{u^{j}_V}_{j} -2\tilde{C}\frac{1}{k^{\alpha}}\norm{u^j_E}_j\\
&\geq \left( e^{\lambda _1-\delta} - e^{\lambda _2+\delta} -2\tilde{C}\frac{1}{k^{\alpha}}\right) \norm{u^j_E}_j.
\end{split}
\end{displaymath}
Taking $k$ large enough so that \[e^{\lambda _1-\delta} - e^{\lambda _2+\delta} -2\tilde{C}\frac{1}{k^{\alpha}}>0\] and applying~\eqref{aux: lower 1.0}
 to the previous inequality,  we conclude that there exists $\gamma >0 $ such that 
\[ \norm{w^{j+1}_E}_{j+1} - \norm{w^{j+1}_V}_{j+1}\geq \gamma \norm{w^{j+1}_E}_{j+1}.\] 
Hence, $w=A(f^j(p))u\in C^{j+1}_{\gamma}$ which yields~\eqref{ineq: cone}. The proof of the lemma is completed. 
\end{proof}

As a simple consequence of Lemma~\ref{lem: main 1},  we obtain the  following result. 

\begin{corollary}\label{cor: growth in Coo}
For every $k\in \mathbb{N}$ large enough,
\begin{equation*}
\lambda (p_k,u)\geq \lambda_1 -3\delta
\end{equation*}
for every $u\in C^{0,1}_0\setminus \{0\}$.
\end{corollary}
\begin{proof}
Recall we are assuming that the Oseledets splitting and the Lyapunov norm are continuous on $\Reg _{\delta ,N}$. In particular, 
if $k$ is sufficiently large (and consequently $x$ and $f^{n_k}(x)$ are close) we have that  
$C^{n_k,1}_{\gamma}\subset C^{0,1}_0$ and thus by \eqref{ineq: cone}, \[A^{n_k}(p_k)(C^{0,1}_0)\subset C^{0,1}_0.\] 
Consequently, for any $u\in C^{0,1}_0$ and $m\in \mathbb{N}$ we have $A^{n_km}(p_k)u \in C^{0,1}_0$. 
Therefore, given $u\in C^{0,1}_0$ and  invoking  \eqref{ineq: Lyap norm cone} and  \eqref{ineq: cone} (together with the fact that the Lyapunov norms 
at $x$ and $f^{n_k}(x)$ are close whenever $k\gg 0$), 
 we obtain that 
\begin{displaymath}
\begin{split}
\norm{A^{n_k}(p_k)u}_{n_k} & \geq \norm{(A^{n_k}(p_k)u)^{n_k}_E}_{n_k}\geq  e^{{n_k}(\lambda _1 -2\delta)}\norm{u^0_E}_0 \\
&\geq \frac{1}{2} e^{{n_k}(\lambda _1 -2\delta)}\norm{u}_0\geq \frac{1}{4} e^{{n_k}(\lambda _1 -2\delta)}\norm{u}_{n_k}.
\end{split}
\end{displaymath}
By iterating, we have that 
\begin{displaymath}
\norm{A^{n_km}(p)u}_{n_k}  \geq \frac{1}{4^m} e^{n_km(\lambda _1 -2\delta)}\norm{u}_{n_k} \quad \text{for $m\in \mathbb N$.}
\end{displaymath}
Consequently, 
\begin{displaymath}
\begin{split}
\lambda (p_k,u)&\geq \lim _{m\to \infty} \frac{1}{n_km}\log \left(\norm{A^{n_km}(p)u}_{n_k} \right) \\
&\geq \lim _{m\to \infty} \frac{1}{n_km} \log \left( \frac{1}{4^m} e^{n_km(\lambda _1 -2\delta)}\norm{u}_{n_k} \right)\\
&=\lambda _1-2\delta- \frac{\log 4}{n_k} +\frac{1}{n_k}\lim _{m\to \infty}\frac{1}{m}\log \left(\norm{u}_{n_k}\right)\\
&\geq \lambda_1 -3\delta,
\end{split}
\end{displaymath}
for $k$ sufficiently  large  which proves our claim. 

\end{proof}

Let $i^k_1=\max\{i; V_i(p_k)\cap C^{0,1}_0\neq \{0\}\}$. Since $V_{i+1}(p_k)\subset V_i(p_k)$ for each $i\in \mathbb N$, we note that 
$V_i(p_k)\cap C^{0,1}_0\neq \{0\}$ for every $i\in \{1, \ldots, i^k_1\}$.

\begin{corollary} \label{cor: l_i0 geq l_1}
We have that 
$$\lambda_{i^k_1}(p_k)\geq \lambda_1-3\delta.$$
\end{corollary}
\begin{proof}
Take $0\neq u\in V_{i^k_1}(p_k)\cap C^{0,1}_0$. It follows from Corollary~\ref{cor: growth in Coo} that $\lambda (p_k,u)\geq \lambda_1 -3\delta$. In particular, $\lambda_{i^k_1}(p_k)\geq \lambda_1-3\delta$ as claimed.
\end{proof}

\begin{corollary}\label{cor: dimenions}
We have that 
\begin{displaymath}
\text{dim}(E_1(p_k)\oplus \ldots \oplus E_{i^k_1}(p_k))= \text{dim}(E_1(x)),
\end{displaymath}
for every $k\gg 0$.
\end{corollary}
\begin{proof}
Let $\hat{d}_{i^k_1}=\text{dim}(E_1(p_k)\oplus \ldots \oplus E_{i^k_1}(p_k))$. By Corollary~\ref{cor: l_i0 geq l_1}, we have  that $\gamma_i(p_k)\geq \gamma_i(\mu)-3\delta$
for every $i\in \{1, \ldots,\hat{d}_{i^k_1}\}$. Therefore, it follows from Lemma~\ref{lem: upper semic} and the choice of $\delta$  that $\hat{d}_{i^k_1}\leq d_1(\mu)$. 
Indeed, suppose $\hat{d}_{i^k_1}> d_1(\mu)$. In particular, $\gamma _{d_1(\mu)+1}(p_k)\geq \lambda_1-3\delta$. Thus, on the one hand we have that
$$\sum _{i=1}^{d_1(\mu)+1} \gamma _{i}(p_k)\geq (d_1(\mu)+1)(\lambda_1-3\delta).$$
On the other hand, by  Lemma~\ref{lem: upper semic} we have that 
$$\sum _{i=1}^{d_1(\mu)+1} \gamma _{i}(p_k)\leq \sum _{i=1}^{d_1(\mu)+1} \gamma _{i}(\mu) +\delta = d_1(\mu)\lambda_1 +\lambda _2 +\delta,$$
for every $k\gg0$. Combining these two inequalities we obtain that
$$(3d_1(\mu)+4)\delta >\lambda _1-\lambda _2,$$
which yields  a contradiction with our  choice of $\delta$. Hence, we conclude that $\hat{d}_{i^k_1}\leq d_1(\mu)$.

In order to obtain the reverse inequality, let us suppose  that  $\hat{d}_{i^k_1}< d_1(\mu)$. Let $\{u_1,\ldots u_{d_1(\mu)}\}$ be a linearly independent subset of $E_1(x)$ and 
write $u_i$ in the form \[u_i=u^i_{p_k}+v^i_{p_k}\quad \text{where $u^i_{p_k} \in E_1(p_k)\oplus \ldots \oplus E_{i^k_1}(p_k)$ and $v^i_{p_k} \in V_{i^k_1+1}(p_k)$,}\] 
for $i=\{1,\ldots , d_1(\mu)\}$. Since $\hat{d}_{i^k_1}< d_1(\mu)$, it follows that $\{u^i_{p_k}\}_{i=1}^{d_1(\mu)}$ is a linearly dependent subset of $E_1(p_k)\oplus \ldots \oplus E_{i^k_1}(p_k)$. Thus, we may assume without loss of generality that 
$$u^1_{p_k}=a_2u^2_{p_k}+\ldots +a_{d_1(\mu)}u^{d_1(\mu)}_{p_k},$$
for some $a_i\in \mathbb{R}$, $i\in \{2,\ldots d_1(\mu)\}$. Consequently, on the one hand we have that 
$$ 0\neq u_1-a_2u_2-\ldots -a_{d_1(\mu)}u_{d_1(\mu)}\in E_1(x) \subset C^{0,1}_0.$$
On the other hand,
$$0\neq u_1-a_2u_2-\ldots -a_{d_1(\mu)}u_{d_1(\mu)}=v^1_{p_k}-a_2v^2_{p_k}-\ldots -a_{d_1(\mu)}v^{d_1(\mu)}_{p_k}\in V_{i^k_1+1}(p_k),$$
contradicting the choice of $i^k_1$. Thus, $\hat{d}_{i^k_1}= d_1(\mu)$ as claimed.
\end{proof}

Now, as a simple consequence of the previous two corollaries we obtain the following result. 
\begin{corollary} \label{cor: app largest multiplicities}
$$\gamma_i(p_k)\geq \gamma_i(\mu)-3\delta $$
for every $i=1, \ldots,d_1(\mu)$ and $k\gg 0$.
\end{corollary}

\subsection{Approximation of the second largest Lyapunov exponent} We proceed in a similar manner to that in  Subsection~\ref{sec: largest LP}.  
For each $1\leq j\leq n_k$, let us consider the splitting  $\B=E_{1}(f^j(x))\oplus E_{2}(f^j(x))\oplus V_3(f^j(x))$.
We can write each $u\in \B$ as 
\begin{equation}\label{repr}u=u^j_{E_1} + u^j_{E_2}+u^j_V, \quad \text{where $u^j_{E_i}\in E_{i}(f^j(x))$ for $i=1,2$ and $u^j_V\in V_3(f^j(x))$}.\end{equation}
For $\gamma \in (0, 1)$,  let us consider the cone $C^{j,2}_{\gamma }$ defined (in terms of the decomposition in~\eqref{repr}) by 
\[
C^{j,2}_{\gamma } =\bigg{\{}u\in \B:   \norm{u^j_V}_{f^j(x)}\leq (1-\gamma) \norm{u^j_{E_2}}_{f^j(x)} \bigg{\}}. 
\]
As before, in order to simplify the notation, we will write $\norm{\cdot}_j$ for the Lyapunov norm at the point $f^j(x)$.

\begin{lemma} \label{lem: main 2}
Let $u\in C^{j,2}_{0}\setminus \{0\}$ for some $j\in \{0,\ldots ,n_k-1\}$. Then, either $u\in C^{j,1}_{0}$ or 
\begin{equation}\label{ineq: Lyap norm cone 2}
\norm{(A(f^j(p_k))u)^{j+1}_{E_2}}_{j+1}\geq e^{\lambda _2 -2\delta}\norm{u^j_{E_2}}_j
\end{equation}
and  
\begin{equation}\label{ineq: cone 2}
A(f^j(p_k))u \in C^{j+1,2}_{\gamma},
\end{equation}
for some $\gamma \in (0,1)$ and every $k$ sufficiently large. Moreover, $k$ and $\gamma$ do not depend on $u$ or $j$.
\end{lemma}

\begin{proof}The proof is similar to the proof of Lemma~\ref{lem: main 1}. 
Suppose that $u\in C^{j,2}_{0}\setminus C^{j,1}_{0}$ since otherwise there is nothing to prove. In particular, $4\norm{u^j_{E_2}}_j \geq \norm{u}_j$. Indeed, since $u\notin C^{j,1}_{0} $,
$$\norm{u^j_{E_1}}_j<\norm{u^j_{E_2}+ u^j_{V}}_j\leq \norm{u^j_{E_2}}_j+\norm{u^j_{V}}_j\leq 2\norm{u^j_{E_2}}_j. $$
Thus, 
\begin{equation}\label{eq: norm u2}
  \norm{u}_j\leq \norm{u^j_{E_1}}_j+\norm{u^j_{E_2}}_j+\norm{u^j_{V}}_j \leq 4 \norm{u^j_{E_2}}_j.
\end{equation}
Let $v=A(f^j(x))u$ and consider $w=A(f^j(p_k))u$. By~\eqref{ineq: Lyapunov norm}, we have that 
\begin{displaymath}
\norm{v^{j+1}_{E_2}}_{j+1}=\norm{A(f^j(x))u^j_{E_2}}_{j+1}\geq e^{\lambda _2 -\delta}\norm{u^j_{E_2}}_j
\end{displaymath}
and
\begin{equation}\label{eq: norm vjF 2}
\norm{v^{j+1}_V}_{j+1}=\norm{A(f^j(x))u^j_V}_{j+1}\leq e^{\lambda _3 +\delta}\norm{u^j_V}_j.
\end{equation}
Moreover, by considering $B_j=A(f^j(p_k))-A(f^j(x))$  we have (as in the proof of Lemma \ref{lem: main 1}) that $w=v+B_ju$ and thus 
\[w^{j+1}_{E_1}=v^{j+1}_{E_1}+ (B_ju)^{j+1}_{E_1}, \ w^{j+1}_{E_2}=v^{j+1}_{E_2}+ (B_ju)^{j+1}_{E_2} \quad  \text{and} \quad w^{j+1}_V=v^{j+1}_V+ (B_ju)^{j+1}_V.\] 
Therefore, using \eqref{eq: norm u2} and proceeding as in Lemma \ref{lem: main 1} we obtain that


$$\norm{B_ju}_{j+1}\leq \tilde{C}\frac{1}{k^{\alpha}}\norm{u^j_{E_2}}_j,$$
for some $\tilde{C}>0$ which is independent of $n_k$ and $j$. Consequently,
\begin{displaymath}
\begin{split}
\norm{w^{j+1}_{E_2}}_{j+1} & \geq \norm{v^{j+1}_{E_2}}_{j+1} -\norm{(B_ju)^{j+1}_{E_2}}_{j+1}\\
&\geq e^{\lambda _2 -\delta}\norm{u^j_{E_2}}_j -\tilde{C}\frac{1}{k^{\alpha}}\norm{u^j_{E_2}}_j \geq  e^{\lambda _2 -2\delta}\norm{u^j_{E_2}}_j,
\end{split}
\end{displaymath}
whenever $k$ is sufficiently large which is precisely inequality \eqref{ineq: Lyap norm cone 2}. In order to obtain~\eqref{ineq: cone 2}, we observe initially that
\begin{equation}\label{aux: lower 1}
\norm{w^{j+1}_{E_2}}_{j+1}\leq  e^{\lambda _2 +\delta}\norm{u^j_{E_2}}_j +\tilde{C}\frac{1}{k^{\alpha}}\norm{u^j_{E_2}}_j \leq  \hat{C}\norm{u^j_{E_2}}_j.
\end{equation}
On the other hand,
\begin{displaymath}
\norm{w^{j+1}_{E_2}}_{j+1}  \geq \norm{v^{j+1}_{E_2}}_{j+1}  -\norm{B_ju}_{j+1}
\end{displaymath}
and
\begin{displaymath}
\norm{w^{j+1}_V}_{j+1}  \leq \norm{v^{j+1}_V}_{j+1}  +\norm{B_ju}_{j+1}. 
\end{displaymath}
By combining the last two  inequalities and using that $u\in C^{j,2}_0$, we have that 
\begin{displaymath}
\begin{split}
\norm{w^{j+1}_{E_2}}_{j+1} - \norm{w^{j+1}_V}_{j+1} & \geq \norm{v^{j+1}_{E_2}}_{j+1} -\norm{v^{j+1}_V}_{j+1}  -2\norm{B_ju}_{j+1}\\
&\geq e^{\lambda _2-\delta}\norm{u^{j}_{E_2}}_{j} - e^{\lambda _3+\delta}\norm{u^{j}_V}_{j} -2\tilde{C}\frac{1}{k^{\alpha}}\norm{u^j_{E_2}}_j\\
&\geq \left( e^{\lambda _2-\delta} - e^{\lambda _3+\delta} -2\tilde{C}\frac{1}{k^{\alpha}}\right) \norm{u^j_{E_2}}_j.
\end{split}
\end{displaymath}
Taking $k$ large enough so that \[e^{\lambda _2-\delta} - e^{\lambda _3+\delta} -2\tilde{C}\frac{1}{k^{\alpha}}>0 \]and applying \eqref{aux: lower 1} to the 
previous inequality,  we conclude  that there exists $\gamma >0 $ such that 
\[\norm{w^{j+1}_{E_2}}_{j+1} - \norm{w^{j+1}_V}_{j+1}\geq \gamma \norm{w^{j+1}_{E_2}}_{j+1},\] which implies that $w=A(f^j(p))u\in C^{j+1,2}_{\gamma}$. Hence, we
conclude that~\eqref{ineq: cone 2} holds and the proof of the lemma is completed.

\end{proof}

\begin{corollary}\label{cor: growth C002}
For every $k\in \mathbb{N}$ large enough,
\begin{equation*}
\lambda (p_k,u)\geq \lambda_2 -3\delta
\end{equation*}
for every $u\in C^{0,2}_0 \setminus \{0\}$.
\end{corollary}

\begin{proof}
Let $k\in \mathbb{N}$ be large enough so that $C^{n_k,2}_{\gamma}\subset C^{0,2}_0$ (recall we are assuming the Oseledets splitting and the Lyapunov norm 
are continuous on $\Reg _{\delta ,N}$ and that $\lim _{k\to +\infty} d(x,f^{n_k}(x))=0$). Thus, it follows from Lemma \ref{lem: main 2} that given $u\in C^{0,2}_0 \setminus \{0\}$, either there exist $m\in \mathbb{N}$ and $j\in \{0,1,\ldots ,n_k-1\}$ so that
$$A^{n_km+j}(p_k)u\in  C^{j,1}_0$$
or 
$$A^{n_km+j}(p_k)u\in  C^{j,2}_0 \setminus  C^{j,1}_0,$$
for every $m\in \mathbb{N}$ and every $j\in \{0,1,\ldots ,n_k-1\}$. 
In the first case, it follows from Lemma \ref{lem: main 1} and Corollary \ref{cor: growth in Coo} that \[\lambda (p_k,u)\geq \lambda_1 -3\delta \geq \lambda_2 -3\delta,\] 
which gives the desired conclusion.

Suppose now that we are in the second case. By recalling~\eqref{ineq: Lyap norm cone 2}, \eqref{ineq: cone 2}  and~\eqref{eq: norm u2} together with the  fact that the Lyapunov norms at $x$ and $f^{n_k}(x)$ are close whenever $k\gg 0$, we obtain that 
\begin{displaymath}
\begin{split}
\norm{A^{n_k}(p_k)u}_{n_k} & \geq \norm{(A^{n_k}(p_k)u)^{n_k}_{E_2}}_{n_k}\geq  e^{{n_k}(\lambda _2 -2\delta)}\norm{u^0_{E_2}}_0 \\
&\geq \frac{1}{4} e^{{n_k}(\lambda _2 -2\delta)}\norm{u}_0\geq \frac{1}{8} e^{{n_k}(\lambda _2 -2\delta)}\norm{u}_{n_k}.
\end{split}
\end{displaymath}
By iterating, we conclude that 
\begin{displaymath}
\norm{A^{n_km}(p)u}_{n_k}  \geq \frac{1}{8^m} e^{n_km(\lambda _2 -2\delta)}\norm{u}_{n_k}.
\end{displaymath}
Consequently, 
\begin{displaymath}
\begin{split}
\lambda (p_k,u)&\geq \lim _{m\to \infty} \frac{1}{n_km}\log \left(\norm{A^{n_km}(p)u}_{n_k} \right) \\
&\geq \lim _{m\to \infty} \frac{1}{n_km} \log \left( \frac{1}{8^m} e^{n_km(\lambda _2 -2\delta)}\norm{u}_{n_k} \right)\\
&=\lambda _2-2\delta- \frac{\log 8}{n_k} +\frac{1}{n_k}\lim _{m\to \infty}\frac{1}{m}\log \left(\norm{u}_{n_k}\right) \\
&\geq \lambda_2 -3\delta,
\end{split}
\end{displaymath}
for any $k$ sufficiently large which proves our claim.

\end{proof}

Let $i^k_2=\max\{i; V_i(p_k)\cap C^{0,2}_0\neq \{0\}\}$.

\begin{corollary}\label{cor: ineq dimensions}
We have that
\begin{displaymath}
\text{dim}(E_1(p_k)\oplus \ldots \oplus E_{i^k_2}(p_k))\geq \text{dim}(E_1(x)\oplus E_2(x)),
\end{displaymath}
for every $k\gg 0$. 
\end{corollary}
\begin{proof} 
The proof is analogous to the second part of the proof of Corollary \ref{cor: dimenions}.


 
\end{proof}

\begin{corollary}We have that
\begin{equation}\label{0547}\gamma_i(p_k)\geq \gamma_i(\mu)-3\delta, \end{equation}
for every $i\in \{1, \ldots,d_1(\mu)+d_2(\mu)\}$ and $k\gg 0$.
\end{corollary}
\begin{proof}
We first note that it follows from Corollary~\ref{cor: app largest multiplicities}  that~\eqref{0547} holds  for every $i\in \{1, \ldots,d_1(\mu)\}$ and $k\gg 0$. 
Now, on the one hand we have that \[\gamma_i(\mu)=\lambda_2, \quad \text{for every $i\in \{d_1(\mu)+1,\ldots, d_1(\mu)+d_2(\mu)\}$.}\] 
On the other hand, Corollary \ref{cor: growth C002} implies  that 
\[\lambda (p_k,u)\geq \lambda_2 -3\delta \quad  \text{for every $u\in V_{i^k_2}\cap C^{0,2}_0 \setminus \{0\}$ and $k\gg 0$,}\] which implies that $\lambda_{i^k_2}(p_k)\geq \lambda_2 -3\delta$. 
By Corollary~\ref{cor: ineq dimensions}, we have that \[ \gamma_i(p_k)\geq \lambda_{i^k_2}(p_k) \quad \text{for every $i\in \{d_1(\mu)+1,\ldots, d_1(\mu)+d_2(\mu)\}$,}\] 
which yields the desired conclusion.
\end{proof}

\begin{corollary}\label{sase}
We have that
\begin{displaymath}
\text{dim}(E_1(p_k)\oplus \ldots \oplus E_{i^k_2}(p_k))= \text{dim}(E_1(x)\oplus E_2(x)),
\end{displaymath}
for every $k\gg 0$. 
\end{corollary}

\begin{proof}
In a view of Corollary~\ref{cor: ineq dimensions}, it is sufficient to prove that
\begin{equation}\label{0608}
 \text{dim}(E_1(p_k)\oplus \ldots \oplus E_{i^k_2}(p_k))\leq  \text{dim}(E_1(x)\oplus E_2(x)), \quad \text{for $k\gg 0$.}
\end{equation}
In order to establish~\eqref{0608}, we adapt the arguments from the proof of
 Corollary \ref{cor: dimenions}. 
Suppose that~\eqref{0608} doesn't hold, i.e. that $\text{dim}(E_1(p_k)\oplus \ldots \oplus E_{i^k_2}(p_k))> d_1(\mu)+d_2(\mu)$. 
In particular, \[\gamma _{d_1(\mu)+d_2(\mu)+1}(p_k)\geq \lambda_2-3\delta.\] Thus, on the one hand we have that
$$\sum _{i=1}^{d_1(\mu)+d_2(\mu)+1} \gamma _{i}(p_k)\geq d_1(\mu)(\lambda_1-3\delta)+(d_2(\mu)+1)(\lambda_2-3\delta).$$
On the other hand,  Lemma~\ref{lem: upper semic} implies that 
$$\sum _{i=1}^{d_1(\mu)+d_2(\mu)+1} \gamma _{i}(p_k)\leq \sum _{i=1}^{d_1(\mu)+d_2(\mu)+1} \gamma _{i}(\mu) +\delta = d_1(\mu)\lambda_1 +d_2(\mu)\lambda _2 + \lambda _3 +\delta,$$
for every $k\gg0$. By combining the last  two inequalities, we obtain that
$$(3d_1(\mu)+3d_2(\mu)+4)\delta >\lambda _2-\lambda _3,$$
which yields a contradiction with our  choice of $\delta$. We conclude that~\eqref{0608} holds and the proof is completed. 
\end{proof}

\subsection{Conclusion of the proof of Theorem~\ref{the: main}}
%
%

More generally, for each $1\leq j\leq n_k$ and $h\in \{1,\ldots ,s\}$, let us consider the splitting  
\[\B=E_{1}(f^j(x))\oplus \ldots \oplus E_{h}(f^j(x))\oplus V_{h+1}(f^j(x)).\] 
We can write each  $u\in \B$ in the form \[u=u^j_{E_1} +\ldots+ u^j_{E_h}+u^j_V, \] 
where $u^j_{E_i}\in E_{i}(f^j(x))$ for $i\in \{1,\ldots, h\}$ and $u^j_V\in V_{h+1}(f^j(x))$. In addition, we can consider cones
\begin{displaymath}
\begin{split}
C^{j,h}_{\gamma }& =\bigg{\{} u\in \B:  \norm{u^j_V}_{f^j(x)}\leq (1-\gamma) \norm{u^j_{E_h}}_{f^j(x)}\bigg{\}},
\end{split}
\end{displaymath}
where $\gamma \in (0, 1)$ and the corresponding numbers $i^k_h=\max\{i; V_i(p_k)\cap C^{0,h}_0\neq \{0\}\}$. By repeating the previous arguments (with straightforward adjustments),  we conclude that
$$\gamma_i(p_k)\geq \gamma_i(\mu)-3\delta, $$
for every $i\in \{1, \ldots,d_1(\mu)+\ldots +d_s(\mu)\}$ and $k\gg 0$. This together with Corollary~\ref{cor: upper semic p_k} implies the conclusion of Theorem~\ref{the: main}.

\section{Applications}\label{APP}
In this section we discuss some applications of the main result of our paper.  We shall mostly restrict our attention to the case of compact cocycles in order to avoid dealing with technicalities. 
\subsection{Uniform hyperbolicity via nonvanishing of Lyapunov exponents}
We begin by recalling that the cocycle $A$ is said to be \emph{uniformly hyperbolic} if there exist a family of projections $P(x)$, $x\in M$ on $\mathcal B$ and constants $D, \lambda>0$
such that:
\begin{enumerate}
 \item for each $x\in M$, we have 
 \begin{equation}\label{UH1}
  A(x)P(x)=P(f(x))A(x)
 \end{equation}
and that the map \begin{equation}\label{UH11} A(x)\rvert \Ker P(x)\colon \Ker P(x)  \to \Ker (P(x)) \quad \text{is invertible}; \end{equation}
\item for each $x\in M$ and $n\ge 0$,
\begin{equation}\label{UH2}
 \lVert A^n(x)P(x)\rVert \le De^{-\lambda n}
\end{equation}
and
\begin{equation}\label{UH3}
 \lVert A^{-n}(x)(\Id-P(x))\rVert \le De^{-\lambda n},
\end{equation}
where \[A^{-n}(x)=(A^n(f^{-n}(x))\rvert \Ker P(f^{-n}(x)))^{-1}\colon \Ker P(x) \to \Ker P(f^{-n}(x)).\]

\end{enumerate}
We note that the condition~\eqref{UH3} can be replaced by the requirement that
\[
 \lVert A^n(x)v\rVert \ge \frac{1}{D}e^{\lambda n}\lVert v\rVert \quad \text{for $n\ge 0$ and $v\in \Ker P(x)\setminus \{0\}$.}
\]
Let $\mathcal E(f)$ denote the set of all $f$-invariant Borel probability measures on $M$ which are ergodic. Furthermore, let $\mathcal E_{per}(f)$ denote 
those measures in $\mathcal E(f)$ whose support is an $f$-periodic orbit. 
\begin{theorem}\label{tuh}
 Assume that $A\colon M \to B(\mathcal B, \mathcal B)$ is an $\alpha$-H\"{o}lder continuous cocycle that takes values in a family of compact operators on $\mathcal B$.
 Furthermore, suppose that there exist a family of projections $P(x)$, $x\in M$ and $\delta >0$ such that:
 \begin{enumerate}
  \item $x\mapsto P(x)$ is a continuous map from $M$ to $B(\mathcal B, \mathcal B)$;
  \item \eqref{UH1} and~\eqref{UH11} hold for each $x\in M$;
  \item for any $\mu \in \mathcal E_{per}(f)$,  we have that 
  \begin{equation}\label{gf}
   \lim_{n\to \infty}\frac 1 n \log \lVert A^n(x)v\rVert \le -\delta \quad \text{and} \quad  \lim_{n\to \infty} \frac 1 n \log \lVert A^n(x)w\rVert \ge \delta,
  \end{equation}
for $\mu$-a.e. $x\in M$ and every $v\in \Ima P(x)$, $w\in \Ker P(x)$.

 \end{enumerate}
Then, the cocycle $A$ is uniformly hyperbolic.
\end{theorem}

\begin{proof}
 We define a sequence of maps $F_n\colon M \to \mathbb R \cup \{-\infty\}$, $n\ge 0$ by
 \[
  F_n(x)=\log \lVert A^n(x)P(x)\rVert, \quad x\in M.
 \]
\begin{lemma}\label{sub}
The sequence $(F_n)_{n\ge 0}$ is subadditive, i.e. 
\[
 F_{n+m}(x)\le F_n(f^m(x))+F_m(x) \quad \text{for every $n, m\ge 0$ and $x\in M$.}
\]
\end{lemma}
\begin{proof}[Proof of the lemma]
By~\eqref{UH1}, we have that
\[
\begin{split}
\lVert A^{n+m}(x)P(x)\rVert &=\lVert A^n(f^m(x))A^m(x)P(x)^2 \rVert \\
&=\lVert A^n(f^m(x))P(f^m(x))A^m(x)P(x)\rVert \\
&\le \lVert A^n(f^m(x))P(f^m(x) \rVert \cdot \lVert A^m(x)P(x)\rVert,
\end{split}
\]
for each $x\in M$ and $n, m\ge 0$. This readily implies the desired conclusion.
\end{proof}
Since both $x\mapsto A(x)$ and $x\mapsto P(x)$ are continuous, we have that $F_n$ is a continuous map for each $n\ge 0$. In particular, $F_1$ is integrable with respect to any $\mu\in \mathcal E(f)$. Hence, 
it follows from Lemma~\ref{sub} and  Kingman's subadditive ergodic theorem that for each $\mu \in \mathcal E(f)$, there exists $\Lambda (\mu)\in [-\infty, \infty)$ 
such that 
\[
 \Lambda (\mu)=\lim_{n\to \infty} \frac{F_n(x)}{n}\quad \text{for $\mu$-a.e. $x\in M$.}
\]
\begin{lemma}\label{0}
 We have that $\Lambda (\mu)$ is either $-\infty$ or  a Lyapunov exponent of the cocycle $A$ with respect to $\mu$. 
\end{lemma}

\begin{proof}[Proof of the lemma]
Assume that $\Lambda(\mu)\neq -\infty$ since otherwise there is nothing to prove. 
Let $\lambda_1>\lambda_2 >\ldots $ denote (distinct) Lyapunov exponents of $A$ with respect to $\mu$. Assuming that $\Lambda(\mu)$ is not a Lyapunov exponent of $A$ with respect to $\mu$, we can 
find $i$ such that $\Lambda(\mu)\in (\lambda_{i+1}, \lambda_i)$. In particular, 
\begin{equation}\label{s}
\begin{split}
 \lim_{n\to \infty}\frac 1 n \log \lVert A^n(x)v\rVert &= \lim_{n\to \infty}\frac 1 n \log \lVert A^n(x)P(x) v\rVert \\
&\le  \lim_{n\to \infty}\frac 1 n \log \lVert A^n(x)P(x)\rVert \\
&=\Lambda(\mu)<\lambda_i, 
\end{split}
\end{equation}
for $\mu$-a.e. $x\in M$ and $v\in \Ima P(x)\setminus \{0\}$. On the other hand, it follows from Theorem~\ref{thm:Oseledets} that for $\mu$-a.e. $x\in M$ and every $v\in \Ima P(x)$, there exists $j\in \mathbb N$ such that
\[
 \lim_{n\to \infty}\frac 1 n \log \lVert A^n(x)v\rVert=\lambda_j,
\]
which together with~\eqref{s} implies that
\begin{equation}\label{s1}
 \lim_{n\to \infty}\frac 1 n \log \lVert A^n(x)v\rVert \le \lambda_{i+1}.
\end{equation}
By~\eqref{s1} and~\cite[Proposition 14.]{FLQ13}, we have that $\Lambda(\mu)\le \lambda_{i+1}$ which yields a contradiction.

\end{proof}
It follows from~\eqref{gf} that all Lyapunov exponents of $A$ with respect to $\mu \in \mathcal E_{per}(f)$ belong to $\mathbb R \setminus (-\delta, \delta)$. This together
with Theorem~\ref{the: main} and Lemma~\ref{0} implies that $\Lambda (\mu) \le -\delta $ for $\mu \in \mathcal E(f)$. Using~\cite[Theorem 1.]{S}, we obtain that 
\[
 \lim_{n\to \infty} \frac{\max_{x\in M} F_n(x)}{n} \le -\delta, 
\]
which readily implies~\eqref{UH2}. One can similarly establish~\eqref{UH3}. Hence, $A$ is uniformly hyperbolic. 
\end{proof}
One can also establish the version of Theorem~\ref{tuh} for quasi-compact cocycles although under additional assumption that $\kappa (\mu)< \Lambda (\mu)$ for each $\mu \in \mathcal E(f)$. 
\begin{remark}
We emphasize that the first results  in the spirit of Theorem~\ref{tuh} are due to Cao~\cite{C}. More precisely, in the particular case of the derivative cocycle $A(x)=Df(x)$ associated to some smooth
diffeomorphism $f$ on a compact Riemmanian manifold $M$, Cao proved that the existence of a continuous and $Df$-invariant splitting 
\[
T_x M=E_x^s \oplus E_x^u \quad \text{for $x\in M$,}
\]
 together with an assumption that for each $\mu \in \mathcal E(f)$ we have 
\begin{equation}\label{as}
    \lim_{n\to \infty}\frac 1 n \log \lVert A^n(x)v\rVert <0 \quad \text{and} \quad  \lim_{n\to \infty} \frac 1 n \log \lVert A^n(x)w\rVert <0,
\end{equation}
for $\mu$-a.e. $x\in M$ and every $v\in E_x^s$, $w\in E_x^u$,  implies  that the cocycle $A$ is uniformly hyperbolic. Hence, in the statement of  Theorem~\ref{tuh} we have required that~\eqref{gf} holds for
$\mu \in \mathcal E_{per}(f)$, while Cao requires that~\eqref{as} holds for any $\mu \in \mathcal E(f)$, although without any type of uniform estimates for Lyapunov exponents as we have in~\eqref{gf}.

The importance of this type of results steems from the fact that nonvanishing of Lyapunov exponents corresponds (in general) to a weaker concept of 
\emph{nonuniform} hyperbolicity (see~\cite{BP07} for detailed discussion). Therefore, it is interesting to see under which additional assumptions, nonvanishing of Lyapunov exponents implies the existence of \emph{uniform} hyperbolic behaviour. For some more recent results in this direction and further references, we refer to~\cite{HPS}.

\end{remark}

\subsection{Sacker-Sell spectrum}
Let us assume that $M$ is compact and connected metric space and that $f\colon M \to M$ is a continuous map. Furthermore, let $A$ be a continuous cocycle over $(M, f)$  of compact and injective (although not necessarily invertible) operators on $\mathcal B$.  For each $\lambda \in \mathbb R$, we can define a new cocycle $A_{\lambda}$ by
\[
 A_{\lambda}(x)=e^{-\lambda} A(x), \quad x\in M.
\]
Finally, set 
\[
 \Sigma =\{\lambda \in \mathbb R: \text{$A_\lambda$ is not uniformly hyperbolic}\}.
\]
The set $\Sigma $ is called the \emph{Sacker--Sell spectrum} of $A$.  It was proved by Magalh\~aes~\cite{LM} (building on the original work of Sacker and Sell~\cite{SS} for cocycles acting on a finite-dimensional space)
that if $f$ has a periodic orbit, we have that:
\begin{enumerate}
 \item $\Sigma  \subset \mathbb R$ is closed;
 \item  $\Sigma=\emptyset$ or $\Sigma (\Lambda)=\cup_{i=1}^k [a_i, b_i]$ for some
 \[
 b_1\ge a_1 > b_2 \ge a_2 > \ldots >b_k\ge a_k,
 \]
or $\Sigma (\Lambda)=\cup_{i=1}^\infty [a_i, b_i]$ for some
 \[
 b_1\ge a_1 > b_2 \ge a_2 > \ldots >b_i\ge a_i >\ldots  \quad \text{such that $\lim_{i\to \infty} a_i=\lim_{i\to \infty} b_i=-\infty$. }
 \]

\end{enumerate}
The following result is due to Schreiber~\cite{S}.
\begin{theorem}\label{ts}
For each $i$, there exist $\mu_1, \mu_2 \in \mathcal E(f)$ such that $a_1$ is the Lyapunov exponent of $A$ with respect $\mu_1$ and $b_1$ is the Lyapunov exponent of $A$ with respect $\mu_2$.
\end{theorem}
We note that for finite-dimensional and invertible cocycles, Theorem~\ref{ts} was first  established by 
  Johnson, Palmer and Sell~\cite{JPS}.  
Let $L(\mu)$ denote the set of all finite Lyapunov exponents of $A$ with respect to $\mu$. 
\begin{corollary}\label{hod}
 Assume further that $A$ is an  $\alpha$-H\"{o}lder  cocycle such that $\Sigma \neq \emptyset$ and that $f$ satisfies Anosov closing property. Then,
 \[
  \partial \Sigma \subset \overline{\bigcup_{\mu \in \mathcal E_{per}(f)}L (\mu)} \quad \text{and} \quad \overline{\bigcup_{\mu \in \mathcal E(f)}L (\mu)} \subset \Sigma.
 \]

\end{corollary}

\begin{proof}\
 The first inclusion is a direct consequence of  Theorems~\ref{the: main} and~\ref{ts}. The second inclusion is proved in~\cite{LM}.
\end{proof}
We are hopeful that Corollary~\ref{hod} could be useful in numerical estimations of $\Sigma$ since it recognizes boundary points of $\Sigma$ as  accumulation points of Lyapunov exponents along periodic orbits (which are easy to estimate).

\subsection{Spectral radius and growth of the cocycle}
In this subsection, $\rho(C)$ will denote the spectral radius of an operator $C\in B(\mathcal B, \mathcal B)$. Furthermore, let us again consider compact, injective and continuous cocycle $A$. The following result is a particular case of~\cite[Theorem 1.4.]{IM}.
\begin{theorem}\label{idm}
 For any $\mu \in \mathcal E(f)$, we have that
 \[
  \limsup_{n\to \infty}\frac 1 n \log \rho(A^n(x))=\lim_{n\to \infty}\frac 1 n \log \lVert A^n(x)\rVert=\lambda_1(\mu) \quad \text{for $\mu$-a.e. $x\in M$.}
 \]

\end{theorem}
We now prove the following result.
\begin{theorem}
 Assume  that $A$ is an  $\alpha$-H\"{o}lder  cocycle  and that $f$ satisfies Anosov closing property. Then,
 \[
   \lim_{n\to \infty}\max_{x\in M} \lVert A^n(x)\rVert^{1/n}=\sup_{(x, p)\in M\times \mathbb N: f^p(x)=x}\rho(A^p(x))^{1/p}.
 \]

\end{theorem}
\begin{proof}

It follows from~\cite[Theorem 1.]{S} and Theorem~\ref{the: main} that 
\begin{equation}\label{u}
 \lim_{n\to \infty} \frac 1 n \max_{x\in M} \log \lVert A^n(x)\rVert =\sup_{\mu \in \mathcal E(f)} \lambda_1(\mu)=\sup_{\mu \in \mathcal E_{per}(f)} \lambda_1(\mu).
\end{equation}
Assume that $\mu \in \mathcal E_{per}(f)$ is supported on an periodic orbit of a point $x\in M$ with period $p$. Then, it follows from Theorem~\ref{idm} that 
\[
\begin{split}
 \lambda_1(\mu)=\lim_{n\to \infty} \frac{1}{np} \log \lVert A^{np}(x) \rVert &=\limsup_{n\to \infty} \frac{1}{np} \log \rho( A^{np}(x)) \\
 &=\limsup_{n\to \infty} \frac{1}{np} \log \rho( (A^{p}(x))^n)\\
 &=\limsup_{n\to \infty} \frac{1}{np} \log (\rho(A^p(x)))^n \\
 &=\frac 1 p \log \rho(A^p(x)).
 \end{split}
\]
Hence, \eqref{u} implies that
\[
\lim_{n\to \infty} \frac 1 n \max_{x\in M} \log \lVert A^n(x)\rVert=\sup_{(x, p)\in M\times \mathbb N: f^p(x)=x} \log \rho(A^p(x))^{1/p}.
\]
Therefore, 
\[
 \lim_{n\to \infty}\max_{x\in M} \log \lVert A^n(x)\rVert^{1/n}=\sup_{(x, p)\in M\times \mathbb N: f^p(x)=x} \log \rho(A^p(x))^{1/p},
\]
which readily yields the desired result. 
\end{proof}
The above   result is interesting since it connects two quantities that exhibit different behaviour under the action of the cocycle: operator norm which is subadditive and spectral radius which  behaves quite badly with respect to composition of operators.

\subsection{Conjugacy between cocycles and Lyapunov exponents}
Assume now that for $i=1, 2$ we are given a cocycle $A_i$ of operators acting on $\mathcal B_i$ and  over a base space  $(M_i, f_i)$. We say that $A_1$ and $A_2$ are \emph{conjugated} if there 
exists an invertible map $h\colon M_1 \to M_2$ and a family of invertible bounded linear operators $L(x) \colon \mathcal B_1 \to \mathcal B_2$, $x\in M_1$ such that:
\begin{enumerate}
 \item \begin{equation}\label{maph} h\circ f_1=f_2 \circ h; \end{equation}
 \item  we have
 \begin{equation}\label{cc}
  A_1(x)=L(f_1(x))^{-1}A_2(h(x))L(x), \quad \text{for each $x\in M_1$.}
 \end{equation}

\end{enumerate}
\begin{remark}
In the context of smooth dynamics, this notion corresponds to the classical notion of conjugacy. Indeed, if $M_1, M_2$ are smooth compact  Riemmanian manifolds and $f_1, f_2$ are smooth diffeomorphisms, then if a differentiable map $h$ satisfies~\eqref{maph}, one can easily conclude that~\eqref{cc} holds with
\[
A(x)=Df_1(x), \quad B(x)=Df_2(x) \quad \text{and} \quad L(x)=Dh(x).
\]
\end{remark}
Observe  that it follows easily from~\eqref{cc} that
\begin{equation}\label{cc1}
 A_1^n(x)=L(f_1^n(x))^{-1}A_2^n(h(x))L(x), \quad \text{for $x\in M_1$ and $n\in \mathbb N$.}
\end{equation}

\begin{theorem}\label{wct}
 Suppose that:
 \begin{enumerate}
\item $A_1 \colon M_1 \to B(\B_1, \B_1)$ and $A_2 \colon M_2 \to B(\B_2, \B_2)$ are cocycles such that $A_1(x)$ is a compact operator for each $x\in M_1$;
 \item $(M_1, f_1)$ satisfies the Anosov closing property;
 \item $A_1$ is an  $\alpha$-H\"{o}lder  cocycle;
  \item $A_2$ is uniformly hyperbolic;
  \item $A_1$ and $A_2$ are conjugated.
 \end{enumerate}
Then, all Lyapunov exponents of $A_1$ are uniformly bounded away from zero.
\end{theorem}

\begin{proof}
Observe that it follows from~\eqref{cc} that $A_2(x)$ is a compact operator for each $x\in M_2$. In addition, observe that $x$ is a periodic point with period $p$ 
 for $f_1$ if and only if $h(x)$ is a periodic point with period $p$ for $f_2$. Furthermore, in this case it follows from~\eqref{cc1} that 
 \begin{equation}\label{cc2}
  A_1^{np}(x)=L(x)^{-1} A_2^{np}(h(x))L(x), \quad \text{for $n\in \mathbb N$.}
 \end{equation}
By~\eqref{cc2}, Lyapunov exponents of $A_1$ with respect to a measure which is supported on the orbit of $x$ are the same as Lyapunov exponents of $A_2$ with respect to a measure
which is supported on the orbit of $h(x)$. Hence, since $A_2$ is uniformly hyperbolic, we have that all Lyapunov exponents of $A_1$ with respect to invariant measures supported on periodic orbits are uniformly bounded away from zero. Then,
Theorem~\ref{the: main} implies that the same holds for all Lyapunov exponents. 
\end{proof}

\begin{remark}
We emphasize that we haven't assumed any type of information regarding the asymptotic behaviour of  maps $x\mapsto \lVert L(x)\rVert$ and 
$x\mapsto \lVert L(x)^{-1}\rVert$. If we  were to assume that those maps are tempered with respect to any invariant measure for $f_1$, we could conclude (see~\cite{BP07}) that Lyapunov 
exponents of cocycles $A_1$ and $A_2$ are the same and therefore the conclusion of Theorem~\ref{wct} would hold trivially. 
\end{remark}

\medskip{\bf Acknowledgements.} 
We would like to thank anonymous referees for their helpful comments that helped us improve the paper. 
The first author was partially supported by a CAPES-Brazil postdoctoral fellowship under Grant No. 88881.120218/2016-01 at the University of Chicago.
The second author 
 was supported in part by the 
 Croatian Science Foundation under the project IP-2014-09-2285 and by the University of
Rijeka under the project number 17.15.2.2.01.

\end{document}